\documentclass[titlepage,reqno,10pt]{amsart}
\usepackage[T1]{fontenc}
\usepackage{amsaddr}
\usepackage[table]{xcolor}
\usepackage{pgf,pgfarrows,pgfnodes,pgfautomata,pgfheaps,pgfshade,hyperref, amssymb,enumerate,amsmath}
\usepackage[pagewise]{lineno}
\usepackage[all]{xy}
\usepackage[capitalize]{cleveref}
\usepackage{mathtools}
\usepackage[shortlabels]{enumitem}
\usepackage{stmaryrd}
\usepackage{bm}
\usepackage{boldline} 
\usepackage{bbm}
\usepackage{aligned-overset}
\usepackage[normalem]{ulem}
\usepackage{bbold}

\newtheorem{theorem}{Theorem}[section]
\newtheorem*{theorem*}{theorem}
\newtheorem{proposition}[theorem]{Proposition}
\newtheorem{lemma}[theorem]{Lemma}
\newtheorem{corollary}[theorem]{Corollary}

\theoremstyle{definition}
\newtheorem{definition}[theorem]{Definition}

\newtheorem{remark}[theorem]{Remark}

\usepackage{euscript,epsfig}
\usepackage{tikz}
\usetikzlibrary{graphs}
\usetikzlibrary[graphs]
\usetikzlibrary{arrows}
\usetikzlibrary{tikzmark}
\usetikzlibrary{positioning}
\usetikzlibrary{decorations.pathreplacing,calligraphy}
\usepackage{tikz-cd}
\usetikzlibrary{decorations.markings}
\usepackage[colorinlistoftodos]{todonotes}

\def\Z{{\mathbb Z}}

\def\N{{\mathbb N}}

\def\T{{\mathbb T}}

\def\cC{{\mathcal C}}

\def\cM{{\mathcal M}}

\def\cW{{\mathcal W}}
\def\cO{{\mathcal O}}

\def \id {{\rm id}}

\def\Per{{\rm Per}}

\newcommand{\set}[1]{\left\{ #1 \right\}}
\newcommand{\tm}{\subseteq}
\newcommand{\mt}{\supseteq}
\newcommand{\kl}{\left(}
\newcommand{\kr}{\right)}

\newcommand{\ol}[1]{\overline{#1}}

\author[M.I. Cortez]{Mar{\'i}a Isabel Cortez}
\address{Facultad de Matem\'aticas\\ Pontificia Universidad Cat\'olica de Chile\\Avda. Vicuña Mackenna 4860, Macul, Chile }
\email{maria.cortez@uc.cl}

\author[J. Drewlo]{Jamal Drewlo}
\address{Faculty of Mathematic and Computer Science\\Friedrich Schiller University Jena\\Inselplatz 5, 07743 Jena, Germany}
\email{jamal.drewlo@uni-jena.de}

\author[J. G{\'o}mez]{Jaime G{\'o}mez}
\address{Mathematical Institute\\
University of Leiden\\
Einsteinweg 55, 2333 CC, Leiden, Netherlands}
\email{j.a.gomez.ortiz@math.leidenuniv.nl}

\author[T. J{\"a}ger]{Tobias J{\"a}ger}
\address{Faculty of Mathematic and Computer Science\\Friedrich Schiller University Jena\\Inselplatz 5, 07743 Jena, Germany}
\email{tobias.jaeger@uni-jena.de}

\begin{document}

\title[Constructing Toeplitz arrays via cut and project schemes]{Constructing Toeplitz arrays via cut and project schemes} 

\date{}

\begin{abstract}
    We show a one-to-one correspondance between Toeplitz arrays over residually finite topological groups and model sets obtained via specific cut and project schemes, built from the odometer associated to the Toeplitz array. 

    As an application, we construct irregular Toeplitz arrays which are extensions of maximal rank $k$ over their maximal equicontinuous factor (given by the associated odometer) and have exactly $k$ different ergodic measures. A modification of the construction also allows to obtain examples with the same measure-theoretic structure, but infinite maximal rank. 
\end{abstract}

\maketitle 

\section{Introduction}

Toeplitz sequences and their associated Toeplitz flows have been a classical object of study in ergodic theory and topological dynamics for many decades. Introduced by Jacobs and Keane \cite{JK69} (preceded by specific examples by Oxtoby \cite{Ox52}), this class of symbolic systems has been investigated by numerous authors. For instance, Oxtoby used them to demonstrate that a topological dynamical system may be minimal without being uniquely ergodic \cite{Ox52}, and Williams later showed that the set of ergodic measures may have any prescribed cardinality \cite{Wi84}. Going further, Downarowicz showed that any Choquet simplex can be realised as the set of invariant measures of a Toeplitz flow \cite{Dow91}, and in collaboration with Lacroix they proved that any subgroup of the circle may be realised as dynamical point spectrum \cite{DL96}. We refer to \cite{Do05} for an overview and to \cite{Dow88,IL94,MP79} for further results illustrating the richness and wide range of possible dynamical behaviour in this class of symbolic dynamical systems.

More recently, efforts have been made to extend these results to shift actions of more general groups. In particular, Toeplitz arrays over residually finite countable groups have been introduced by Krieger \cite{Kr10} and studied further by Cortez and Petite \cite{CoPe08,CoPe14} (see also \cite{CeCoGo23,CoGo24} for the most recent advances). Apart from their intrinsic interest, such endeavours are well-motivated in the context of aperiodic order and the mathematical theory of quasicrystals. Here, Toeplitz sequences serve as prototypical examples of aperiodic structures with strong long-range order, and higher-dimensional models are much more natural from the physics viewpoint. 
An additional link between the two topics is given by work of Baake, J\"ager and Lenz~\cite{BaJäLe16}, who showed that Toeplitz sequences can also be obtained as model sets via the cut and project method. The latter was originally introduced by Meyer \cite{Mey72} and has since become a fundamental tool for describing Delone sets that model quasicrystals (see \cite{BaLeMo07, DJL, JLO19, LP03, Mo97}). 

The aim of this article is two-fold: first, we extend the result in \cite{BaJäLe16} on the representation of Toeplitz sequences by model sets to the setting of Toeplitz arrays over residually finite countable groups. Thereby, we work in the more general setting of almost automorphic extensions of metric compactifications of the acting group. Secondly, we use recent methods developed for the construction of irregular model sets \cite{FuGlJäOe21, JLO19} in order to provide novel examples of such Toeplitz arrays. \smallskip

To be more precise, we consider cut and project schemes of the form \((G, H, \mathcal{L})\), where \(G\) is an infinite countable group, \(H\) is a compact metrizable group, and the lattice \(\mathcal{L}\) is given by  $\{(g, \tau(g)) : g \in G\}$ with an injective homomorphism \(\tau: G \to H\) whose image is dense in \(H\). In this setting, given $W\subseteq H$, we say that $x\in \{0,1\}^G$ is a model set with respect to the window $W$ if $\{g\in G: x(g)=1\}$ coincides with $\{g\in G: \tau(g)\in W\}$. 
Then we establish a bijective correspondence between almost automorphic subshifts in $\{0,1\}^G$ whose maximal equicontinuous factor is $H$, and model sets associated with windows in $H$ that have special properties. Recall that an element $x \in \{0,1\}^G$ is called {\it almost automorphic on $H$} if there exists a factor map $\beta$, from the orbit closure of $x$ to the equicontinuous system defined by the left action of $G$ on $H$, such that $\beta^{-1}(\{1_H\}) = \{x\}$. Here $1_H$ denotes the unit element of a group $H$. A subset (\textit{window}) $W\subseteq H$ is called \textit{proper} if $\overline{\mathrm{int}(W)}=W$, \textit{irredundant} if $Wh=W$ implies $h=1_H$ and \textit{generic (with respect to $\tau(G)$)} if $\partial W\cap \tau(G)=\emptyset$.  \newpage

\begin{theorem}\label{thm: chara almost automorphic (introduction)}
  Let $G$ be a countable, discrete group, $(H, \tau)$ a metrizable compactification of $G$, and let $\mathcal{W}(H)$ denote the collection of all windows $W \subseteq H$ that are proper, irredundant, and generic. Then the map
\[
\Lambda: \mathcal{W}(H) \to \left\{x \in \{0,1\}^G : x \text{ is almost automorphic on } H\right\}
\]
defined by
\[
\Lambda(W)(g) = 1 \iff \tau(g) \in W, \quad \text{for every } g \in G \text{ and } W \in \mathcal{W}(H),
\]
is well defined and bijective.

Moreover, if $\beta: \overline{O_G(x)} \to H$ is the factor map satisfying $\beta^{-1}(\{1_H\}) = \{x\}$, then the window $W$ associated with $x$ is given by
\[
W = \left\{ \xi \in H : \text{ there exists } y \in \beta^{-1}(\{\xi\}) \text{ such that } y(1_G) = 1 \right\},
\]
and its boundary is
\[
\partial W = W \setminus \left\{ \xi \in H : \forall y \in \beta^{-1}(\{\xi\}), \; y(1_G) = 1 \right\}.
\]
\end{theorem}

As a consequence, we characterize the Toeplitz subshifts \(X \subseteq \{0,1\}^G\) whose maximal equicontinuous factor is a $G$-odometer \(\overleftarrow{G}\) with group structure, as the model sets arising from irredundant, generic, and proper windows in \(\overleftarrow{G}\). 

\begin{corollary}\label{cor: chara Toeplitz model (introduction)}
    Let $G$ be a residually finite, discrete, countable group and $(\overleftarrow{G},G)$ be a free odometer such that $\overleftarrow{G}$ has a group structure. The following statements are equivalent:
\begin{enumerate}
\item The array $x\in\{0,1\}^G$ is a Toeplitz array such that $(\overleftarrow{G},G)$ is the maximal e\-qui\-con\-ti\-nuous factor of $(\overline{O_G(x)},G)$.


\item The array $x\in\{0,1\}^G$ is the model set associated to a proper, irredundant, generic window $W\subseteq \overleftarrow{G}$.
\end{enumerate}
\end{corollary}

This generalizes the characterization of Toeplitz \(\mathbb{Z}\)-subshifts given in \cite{BaJäLe16} to the action of residually finite groups. 
\medskip 

As an application of Theorem~\ref{thm: chara almost automorphic (introduction)}, we adapt and refine methods from \cite{FuGlJäOe21, JLO19} in order to provide novel examples of Toeplitz arrays. These can be seen as a generalisation of William's construction for the case $G=\Z$ in \cite{Wi84}.  

\begin{theorem} \label{thm: existence of k-1 Toeplitz arrays}
Let \( G \) be a countably infinite, amenable and residually finite group. Consider a  $G$-odometer \( \overleftarrow{G} \) with group structure, equipped with its normalized Haar measure $\nu$. Then, for every integer \( k \geq 2 \), there exist an irregular Toeplitz array \( x_0 \in \{0,1\}^G \)  and an almost one-to-one  factor map $\beta: \overline{O_G(x_0)} \to \overleftarrow{G}$ verifying $\beta^{-1}(\{1_{\overleftarrow{G}}\})=\{x_0\}$, such that
\begin{enumerate}
\item For every $\xi\in \overleftarrow{G}$ the fiber $\beta^{-1}(\{\xi\})$ contains at most $k$ elements.
\item For $\nu$-almost every $\xi\in \overleftarrow{G}$ the fiber $\beta^{-1}(\{\xi\})$ contains exactly $k$
elements.
\item The resulting subshift \( (\overline{O_G(x_0)}, G) \) admits exactly \( k \) ergodic invariant probability measures, each of which is measure-theoretically isomorphic to the Haar measure \( \nu \) on \( \overleftarrow{G} \) via the map $\beta$.
\end{enumerate}
\end{theorem}
We note that examples of Toeplitz arrays with a similar measure-theoretic structure have been constructed in \cite{CeCoGo23}. However, the techniques we employ to prove Theorem~\ref{thm: existence of k-1 Toeplitz arrays} differ entirely from those in \cite{CeCoGo23}. In particular, they allow for a precise control of the fibre structure and cardinality, and in contrast to \cite{CeCoGo23} our approach does not require any modification of the alphabet size to construct Toeplitz subshifts with the desired properties. Moreover, a modification of these new techniques allow us to construct Toeplitz subshifts with the same measure-theoretic structure as in the previous theorem, but with almost all fibers being infinite, as stated in the following result. 

\begin{theorem} \label{thm: existence of k-1 Toeplitz arrays with infinite max rank}
    Let \( G \) be a countably infinite, amenable and residually finite group. Consider a  $G$-odometer \( \overleftarrow{G} \) with group structure, equipped with its normalized Haar measure $\nu$. Then, for every $k\geq 2$, there exists an irregular Toeplitz array $x_0\in \{0,1\}^G$ and an almost one-to-one factor map $\beta: \ol{O_G(x_0)}\to \overleftarrow{G}$ which still satisfies (3) as in Theorem~\ref{thm: existence of k-1 Toeplitz arrays}, but has (countably) infinite maximal rank. Moreover, \mbox{$\nu$-almost} every fiber is of countably infinite cardinality. 
\end{theorem}
{\textit{ Structure of the article.}}\quad We first introduce basic notions and preliminaries in Section~\ref{Background}. In Section~\ref{odometers} we discuss $G$-odometers for countably infinite residually finite groups $G$, while Section~\ref{sec: Toeplitz} provides the required background on Toeplitz arrays and their relation to $G$-odometers. In Section~\ref{sec: almost auto as model set}, we then turn to cut and project to provide the proof of Theorem~\ref{thm: chara almost automorphic (introduction)}. As mentioned, we work in the abstract setting of metric compactifications of groups in this context.
However, since a $G$-odometer with group structure naturally provides a metric compactification of $G$, the application to Toeplitz flows is immediate. In Section~\ref{sec: self similarity}, we introduce the concept of self-similar windows of cut-and-project scheme and show how this allows to obtain a precise control of the fibre structure and cardinality of the resulting model sets. This is applied in Section~\ref{sect: self-similar windows}, where we provide an explicit construction of windows with the  respective self-similarity properties and thereby provide a proof of Theorem~\ref{thm: existence of k-1 Toeplitz arrays}. The necessary modifications in order to obtain Theorem~\ref{thm: existence of k-1 Toeplitz arrays with infinite max rank} are then discussed in Section~\ref{sect: infinite max rank}.

\section{Definitions and background}\label{Background}
By a \textbf{dynamical system} we mean a pair $(X, G)$, where $X$ is a compact metric space and $G$ is a countably infinite group acting on $X$ by homeomorphisms. We say that $(X,G)$ is a \textbf{Cantor system} if $X$ is a Cantor set. For each $g \in G$ and $x \in X$, the action of $g$ on $x$ is denoted by $g \cdot x$ (or simply $gx$). Given a subset $A \subseteq X$, we define $g \cdot A = \{g \cdot x : x \in A\}$. A subset $A \subseteq X$ is said to be \textbf{$G$-invariant} if $g \cdot A = A$ for every $g \in G$. The \textbf{orbit} of a point $x \in X$ under the action of $G$ is the set $O_G(x) = \{g \cdot x : g \in G\}$. The dynamical system $(X, G)$ is said to be \textbf{minimal} if the orbit $O_G(x)$ is dense in $X$ for every $x \in X$. Equivalently, minimality means that the only non-empty closed $G$-invariant subset of $X$ is $X$ itself. The dynamical system $(X,G)$ is \textbf{free} if the stabilizer of every point is trivial i.e., $gx=x$ implies $g=1_G$ for every $x\in X$.

 An \textbf{invariant probability measure} of $(X,G)$ is a Borel probability measure $\mu$ on $X$ such that $\mu(g\cdot A)=\mu(A)$, for every $g\in G$ and every Borel subset $A$ of $X$. We denote by $\mathcal{M}(X,G)$ the set of all invariant probability measures of $(X,G)$. It is known that  $G$ is amenable if and only if for every dynamical system $(X,G)$ the set $\mathcal{M}(X,G)$ is non-empty (see, for example, \cite{KL16}). If $\mathcal{M}(X,G)$ has only one element, we say that $(X,G)$ is uniquely ergodic.

Let $(X, G)$ and $(Y, G)$ be dynamical systems. A \textbf{factor map} from $(X, G)$ to $(Y, G)$ is a continuous surjective map $\varphi: X \to Y$ such that $\varphi(gx) = g\varphi(x)$ for all $ x \in X$  and  $g \in G$. The systems $(X, G)$ and $(Y, G)$ are said to be \textbf{conjugate} if there exists a bijective factor map $\varphi: X \to Y$. In this case, $\varphi$ is called a \textbf{conjugacy}. A factor map $\varphi: X \to Y$ is said to be \textbf{almost one-to-one} if the set $\{x \in X : \varphi^{-1}(\{\varphi(x)\}) = \{x\}\}$ is dense in $X$. In that case, we say that $(X, G)$ is an \textbf{almost one-to-one extension} of $(Y, G)$.

\subsection{Almost automorphic systems and equicontinuity}\label{sec_equicontinuous} A dynamical system $(X, G)$ is said to be \textbf{equicontinuous} if for every $\varepsilon > 0$, there exists $\delta > 0$ such that whenever $d(x, y) \leq \delta$, we have $d(g \cdot x, g \cdot y) \leq \varepsilon$  for all  $g \in G.$ Every minimal equicontinuous system $(X, G)$ is conjugate to a system of the form $(E/F, G)$, where $E$ is a compact topological group, $F$ is a closed subgroup of $E$, and there exists a homomorphism $\varphi: G \to E$ such that $\varphi(G)$ is dense in $E$. The action of $G$ on the quotient space $E/F$ is given by
$g \cdot (hF) = \varphi(g) h F,$ for all  $g \in G$  and  $hF \in E/F$ (see~\cite[Theorem 3, Chapter 3]{Au88}). Every topological dynamical system $(X,G)$ on a compact metric space has a  \textbf{maximal equicontinuous factor (MEF)}, that is, a factor $(Y,G)$ with factor map $\pi:X\to Y$ such that for every other equicontinuous factor $(Z,G)$ of $(X,G)$ with factor map $\varphi:X\to Z$ there exists a factor map $\psi:Y\to Z$ such that $\varphi=\psi\circ \pi$ \cite{Au88,Do05}. 

A \textbf{metric compactification} of a group $G$ is a pair $(H, \tau)$, where $H$ is a compact metrizable group and $\tau: G \to H$ is an injective homomorphism such that $\tau(G)$ is dense in $H$. We will also refer to $H$ as a metric compactification of $G$. Without loss of generality, we may assume that $H$ is equipped with a bi-invariant metric inducing its topology; that is, a metric $d$ satisfying $d(ha, hb) = d(a, b) = d(ah, bh)$, for all  $a, b, h \in H$ (see~\cite[Proposition 8.43]{HofMo20}).
The (discrete countable) groups that admit metric compactifications are called \textbf{maximally almost periodic}, and are precisely those that admit (minimal) free, equicontinuous dynamical systems. Indeed, if $(H, \tau)$ is a metric compactification of $G$, then $G$ acts naturally on $H$ by left multiplication, defined by $g \cdot h = \tau(g)h$, for all $g \in G,\, h \in H$. Since each $g \in G$ acts as an isometry, the resulting dynamical system $(H, G)$ is equicontinuous. The density of the orbit of the identity element $1_H$ ensures minimality, and the injectivity of $\tau$ guarantees that the action is free. Conversely, if $(X, G)$ is a minimal, equicontinuous, and free dynamical system, then the associated Ellis semigroup (which is a group in this case) provides a metric compactification of $G$ (see~\cite{Au88}).

A group $G$ is said to be \textbf{residually finite} if, for every non-identity element $g \in G$, there exists a finite group $F$ and a homomorphism $\phi: G \to F$ such that $\phi(g) \neq 1_F$. It is known that countable residually finite groups are precisely those that admit totally disconnected metric compactifications (see~\cite[Lemma~2.1]{CeCoGo23}).  On the other hand, it is known that a finitely generated group is maximally almost periodic if and only if it is residually finite.

A topological dynamical system is said to be \textbf{almost automorphic} if its maximal equicontinuous factor is minimal and the associated factor map $\pi$ is almost one-to-one. Every almost automorphic system contains a unique minimal system (see \cite{FK18}). In general, if $(X,G)$ is any dynamical system, a point $x\in X$   is referred to as \textbf{almost automorphic on $Y$} if 
the maximal equicontinouous factor  of $(\overline{O_G(x)},G)$ is $(Y,G)$, and the fiber of $x$ under the factor map $\pi$ is trivial; that is, $\pi^{-1}(\{\pi(x)\}) = \{x\}$.  

In this work, we concentrate on almost automorphic systems $(X,G)$ whose maximal equicontinuous factor arises from the natural action of the group $G$ on a metric com\-pac\-ti\-fi\-ca\-tion of itself.

\section{Odometers.}\label{odometers}
Let $G$ be an infinite countable group. The group $G$ is residually finite if and only if there exists a decreasing sequence $(\Gamma_n)_{n \in \mathbb{N}}$ of finite-index subgroups of $G$ whose intersection is trivial. Moreover, since every finite-index subgroup contains a finite-index normal subgroup, it follows that $G$ is residually finite if and only if there exists such a sequence $(\Gamma_n)_{n \in \mathbb{N}}$ in which each $\Gamma_n$ is normal (see \cite{CeCo18} for more details).

 \subsection{Representation of odometers} \label{subsect: Representation of odometers}
 
Let $(\Gamma_n)_{n\in\mathbb{N}}$ be a decreasing sequence of finite index subgroups of $G$. 
The {\bf $G$-odometer} associated to $(\Gamma_n)_{n\in\mathbb{N}}$ is defined as the inverse limit of the inverse system $(G/\Gamma_n,\phi_n)$, where $\phi_n:G/\Gamma_{n+1}\to G/\Gamma_n$ is the canonical projection, for every $n\in \mathbb{N}$. In other words, we let 
\begin{align*}
    \overleftarrow{G}:=\varprojlim_n(G/\Gamma_n,\phi_n)=\left\{\left.(g_n\Gamma_n)_{n\in\mathbb{N}}\in\prod_{n\in\mathbb{N}}G/\Gamma_n\ \right| \:\forall n \in \N:\: g_{n+1}\Gamma_n=g_n\Gamma_n\right\},
\end{align*}
where $G/\Gamma_n$ denotes the set of left cosets of $\Gamma_n$ in $G$, for every $n\in\mathbb{N}$.
If every set $G/\Gamma_n$ is endowed with the discrete topology and $\prod_{n\in\mathbb{N}}G/\Gamma_n$ with the product topology, we obtain that $\overleftarrow{G}$ with the topology inherited from $\prod_{n\in\mathbb{N}}G/\Gamma_n$ is metrizable and totally disconnected. 
Note that $\overleftarrow{G}$ is a Cantor set when it is infinite, and this is the case when there exists an increasing sequence $(n_k)_{k\in\mathbb{N}}$ in $\mathbb{N}$ such that $\Gamma_{n_{k+1}}$ is strictly contained in $\Gamma_{n_k}$, for every $k\in \mathbb{N}$. The topology in $\overleftarrow{G}$ is induced by the metric $d$ defined as follows: for $x=(g_n\Gamma_n)_{n\in\mathbb{N}}$ and $y=(h_n\Gamma_n)_{n\in\mathbb{N}}$ in $\overleftarrow{G}$,   we set $d(x,y)=0$ if $x=y$, otherwise
\begin{equation}\label{eq: metric odometer}
d(x,y)=\frac{1}{2^{\min\{k\in\mathbb{N}: g_k\Gamma_k\neq h_k\Gamma_k\}}}.
\end{equation}
According to this metric, for every $k\geq 0$, the ball of radius $\frac{1}{2^k}$ and center $x=(g_n\Gamma_n)_{n\in\mathbb{N}}$ in $\overleftarrow{G}$ is given by
$$
B_{\frac{1}{2^k}}(x)=\{(h_n\Gamma_n)_{n\in\mathbb{N}}\in \overleftarrow{G}: h_k\Gamma_k=g_k\Gamma_k\}.
$$
We also use the notation $[x]_k$ instead of $B_\frac{1}{2^k}(x)$ and refer to these sets as \textbf{cylinder sets}.

The group $G$ acts on $\overleftarrow{G}$ by pointwise left multiplication, i.e., $g\cdot(g_n\Gamma_n)_{n\in\mathbb{N}}=(gg_n\Gamma_n)_{n\in\mathbb{N}}$ for every $g\in G$ and $(g_n\Gamma_n)_{n\in\mathbb{N}} \in \overleftarrow{G}$. Observe that the homeomorphism induced by the action of $g\in G$ is an isometry with respect to $d$, which implies that the system $(\overleftarrow{G},G)$ is equicontinuous. Furthermore, since the orbit of $x_0=(1_G\Gamma_n)_{n\in\mathbb{N}}$ is dense in $\overleftarrow{G}$, we conclude that the system $(\overleftarrow{G},G)$ is minimal (see \cite[Page 37, Lemma 3]{Au88}). The minimal equicontinuous system $(\overleftarrow{G},G)$ is also referred to as \textbf{$G$-odometer}. If $\mu$ is an invariant probability measure of $(\overleftarrow{G},G)$, then  
\begin{equation}\label{measure_definition}
\mu\left(B_{\frac{1}{2^k}}(x) \right)=\frac{1}{[G:\Gamma_k]}, \mbox{ for every } k\geq 0 \mbox{ and } x\in \overleftarrow{G}.
\end{equation}
Since every clopen subset of $\overleftarrow{G}$ is a finite disjoint union of these balls, we deduce that $\mu$ is completely determined by   (\ref{measure_definition}). On the other hand, every measure $\mu$ that satisfies (\ref{measure_definition}) extends to a unique invariant probability measure, showing  that $(\overleftarrow{G},G)$ is uniquely ergodic (see \cite{CoPe08} for more details). The infinite $G$-odometers correspond exactly to minimal equicontinuous actions of $G$ on the Cantor set (see \cite[Proposition 12]{CoGo24}).
 


If we assume that every $\Gamma_n$ is a normal subgroup, then every $G/\Gamma_n$ is a group, which implies that $\overleftarrow{G}$ is a subgroup of $\prod_{n\in\mathbb{N}}G/\Gamma_n$. In this case, the metric $d$ defined above is bi-invariant and the map $\tau:G\to \overleftarrow{G}$ given by $\tau(g)=(g\Gamma_n)_{n\in\mathbb{N}}$ for every $g\in G$, is a group homomorphism such that $\tau(G)$ is dense in $\overleftarrow{G}$. If in addition $\bigcap_{n\in\mathbb{N}}\Gamma_n=\{1_G\}$, then $\tau(G)$ is isomorphic to $G$ and  we call the odometer $(\overleftarrow{G},G)$ {\bf free}.
In this case, when $\overleftarrow{G}$ has group structure and $\tau$ is one-to-one, then ($\overleftarrow{G},\tau)$ is a  totally disconnected metric compactification of $G$. Conversely, every totally disconnected metric compactification of $G$ is a $G$-odometer with group structure, as shown in \cite{CeCoGo23}.

\begin{proposition}[{\cite[Lemma 2.1]{CeCoGo23}}]\label{prop: totally disconnected compactification is odometer}
    Let $G$ be a residually finite group. 
    If $\overleftarrow{G}$ is a totally disconnected metric compactification of $G$, then there exists a decreasing sequence $(\Gamma_n)_{n\in\mathbb{N}}$ of normal subgroups of finite index of $G$ with trivial intersection such that $\overleftarrow{G}$ is conjugate to the $G$-odometer associated to $(\Gamma_n)_{n\in\mathbb{N}}$. 
\end{proposition}


\subsection{An analogue of $p$-adic expansion for residually finite groups} \label{subsect: expansion}
Recall that for a fixed number $p\in\N$, any positive integer $n\in\N$ has a unique $p$-adic expansion 
\begin{equation}\label{e.p-adic_expansion}
n = \sum_{j=1}^\infty n_jp^{j-1},
\end{equation} where $n_j \in \set{0,\ldots,p-1}$ and $n_j=0$ for all but finitely many $j$. If we set $D_n = \set{0,\ldots, p^n-1}$ and $\Gamma_n = p^n\Z$, then $D_j$ is a fundamental domain of $\Z/\Gamma_j$ and $\pi_j(n)=n_jp^{j-1}$is an element of $D_j \cap \Gamma_{j-1}$. The purpose of this subsection is to provide an analogue of this expansion in the setting of residually finite countable groups, which will be needed in the later sections. 
In particular, we will need to have a good understanding of the respective carry-over rule, which is slightly more intricate in the non-commutative setting. 
\medskip 

Suppose that $G$ is an infinite, countable, amenable, residually finite group and let $(\Gamma_n)_{n\in\N}$ be a sequence of normal finite index subgroups $\Gamma_n\leq G$ such that $\bigcap_{n\in\N} \Gamma_n = \{1_G\}$.
We set $\Gamma_0=G$. Our aim is to establish a unique extension of group elements $g\in G$ of the form  ``$g = \prod_{j=1}^\infty \pi_j(g)$'' with $\pi_j(g) \in D_j \cap \Gamma_{j-1}$, where the $D_j$ are suitably chosen fundamental domains of $G/\Gamma_j$. Thereby, we rely on the following result from \cite{CeCoGo23}.

\begin{lemma}[{\cite[Lemma 2.9]{CeCoGo23}}]\label{lem: ex. of seq. of fundamental domains}
Let $G$ be an infinite residually finite group, and let $(\Gamma_n)_{n\in \mathbb{N}}$ be a   decreasing sequence of finite index normal subgroups of $G$ with trivial intersection.     Then there exist  an increasing sequence $(n_i)_{i\in \mathbb{N}}\subseteq \mathbb{N}$ and a sequence $(D_i)_{i\in \mathbb{N}}$  of finite subsets of $G$, such that for every $i\in \mathbb{N}$
\begin{enumerate}
    \item[$\mathrm{(D1)}$] $1_G\in D_1$ and $D_i\subseteq D_{i+1}$.
    \item[$\mathrm{(D2)}$] $D_i$ contains exactly one element of each class in $G/\Gamma_{n_i}$.
    \item[$\mathrm{(D3)}$] $G= \bigcup_{i\in \mathbb{N}} D_i$.
  \item[$\mathrm{(D4)}$] $D_{i+1}=\bigcup_{v\in D_{i+1}\cap \Gamma_{n_j}}D_jv$, for every $j\leq i$.  

  \end{enumerate}
  \end{lemma}

  \begin{remark}\label{rem: D_n Folner if G amenable}
      When the group is amenable, the  sequence $(D_n)_{n\in \mathbb{N}}$ can be also assumed to be a left F\o lner sequence (see \cite[Lemma 5]{CoPe14}), by going over to a subsequence if necessary. However, this fact will only be used in the later sections. 
  \end{remark}
  For the remainder of this section, we assume that the sequence $(D_i)_{i\in \mathbb{N}}$ satisfies the assertions  (D1)--(D4) of Lemma~\ref{lem: ex. of seq. of fundamental domains}. Given $g_1,\ldots,g_n\in G$, we write $\prod_{j=1}^n g_j = g_1 g_2\ldots g_n$. Similarly, if $(g_j)_{j\in\N}$ is a sequence in $G$ such that $g_j=1_G$ for all $j>n$, we define the infinite product as $\prod_{j\in\N} g_j = \prod_{j=1}^n g_j$. Further, we let $D_0=\{1_G\}$. \medskip
  
Let $g\in G$ and $n\in \N$ be fixed.
Then there is a unique decomposition $g = \psi_n(g)\varphi_n(g)$ with $\psi_n(g)\in D_n$ and $\varphi_n(g)\in \Gamma_n$. Note that since $D_0=\{1_G\}$ and $\Gamma_0=G$, we have $\psi_0\equiv 1_G$ and $\varphi_0=\id_G$. Further, if $j\leq n$, then 
\begin{equation} \label{eq: expansion basic step} g = \psi_n(g)\varphi_n(g) =\underbrace{\psi_j(\psi_n(g))}_{\in D_j}\cdot \underbrace{ \varphi_j(\psi_n(g))\varphi_n(g)}_{\in \Gamma_j}.
\end{equation}
By uniqueness of the decomposition, this shows that $\psi_j(\psi_n(g)) = \psi_j(g)$.
In particular, we obtain $$g =\psi_{n-1}(g)  \varphi_{n-1}(\psi_n(g)) \varphi_n(g).$$

Now, due to (D3), there exists a minimal integer $N(g)\geq 0$ such that $g\in D_{N(g)+1}$. We then have $\psi_j(g)=g$ and $\varphi_j(g)=1_G$ for all $j>N(g)$. By the iterated application of (\ref{eq: expansion basic step}), we obtain an expansion 
$$ g = \left(\prod_{j=1}^{N(g)} \varphi_{j-1}(\psi_j(g))\right)\varphi_{N(g)}(g) $$
with $\varphi_{N(g)}(g)\in\Gamma_{N(g)}$ and $\varphi_{j-1}(\psi_j(g))\in\Gamma_{j-1}$ for all $j=1,\ldots,N(g)$. 
We define
\begin{equation} \label{eq: expansion coefficients}
\pi_j(g) := \varphi_{j-1}(\psi_j(g))  
\end{equation}
for all $j=1,\ldots,N(g)+1$ and $\pi_j(g)=1_G$ for all $j>N(g)+1$. Note that we still have $\pi_j(g)=\varphi_{n-1}(\psi_j(g))$ in the latter case, so that (\ref{eq: expansion coefficients}) actually holds for all $j\in\N$. Further, combining the definition of $\varphi_{j-1}$, the fact that $\psi_j(g)\in D_j$ and property (4), we obtain $\pi_j(g)\in D_j\cap \Gamma_{j-1}$. Moreover, we have $\pi_{N(g)+1}(g)=\varphi_{N(g)}(\psi_{N(g)+1}(g))=\varphi_{N(g)}(g)$ and
$$
g=\prod_{j=1}^{N(g)+1} \pi_j(g) = \prod_{j\in\N} \pi_j(g). $$
Thus, the elements $\pi_j(g)\in G$ play exactly the same role as the terms $n_jp^{j-1}$ in equation~(\ref{e.p-adic_expansion}).
\begin{remark}
In principle, we could also work with the finite expansions $g=\prod_{j=1}^{N(g)} \pi_j(g)$ instead of the infinite products. In the end this is a matter of taste, but we find the latter option notationally more convenient for our purposes. 
\end{remark}
The following lemma ensures the uniqueness of the above expansion. 
\begin{lemma} \label{eq: uniqueness of expansion}
    Suppose that 
    $$
    g=\prod_{j=1}^n b_j
    $$
    for some $n\in\N$ and $b_j\in D_j\cap \Gamma_{j-1}$. Then $n\geq N(g)+1$, $b_j=1_G$ for all $j>N(g)+1$ and $b_j=\pi_j(g)$ for all $j=1,\ldots , N(g)+1$. 
\end{lemma}
\begin{proof}
First of all, if $n\leq N(g)$ then $g\in D_{N(g)}$ by (D4), contradicting the definition of $N(g)$. Hence, we have $n\geq N(g)+1$. 
Due to $b_j\in D_j\cap \Gamma_{j-1}$ in combination with property (4), we have that 
$$
\prod_{j=1}^k b_j \in D_k \quad \textrm{ and } \quad \prod_{j=k+1}^n b_j \in \Gamma_k$$
for all $k=1,\ldots,n$. Due to the uniqueness of the decomposition in (\ref{e.p-adic_expansion}), this implies 
$$
\prod_{j=1}^k b_j =\psi_k(g) \quad \textrm{ and } \quad \prod_{j=k+1}^n b_j =\varphi_k(g) \ . 
$$
This immediately yields $b_k=\psi_{k-1}(g)^{-1}\cdot \psi_k(g)$ for all $k=1,\ldots, n$.
However, exactly the same arguments apply to the sequence $\pi_1(g),\ldots \pi_{n}(g)$, so that we obtain  
\begin{equation}\label{eq: expansion alternative}
b_j=\pi_j(g)=\psi_{j-1}^{-1}(g)\cdot \psi_j(g) 
\end{equation}
for all $j=1,\ldots,n$. Since $\pi_{N(g)+2}(g)=\ldots=\pi_N(g)=1_G$, this proves the assertion. 
\end{proof}
\begin{remark}
We note that (\ref{eq: expansion alternative}) would provide an alternative way to define the expansion coefficients $\pi_j(g)$. 
\end{remark}

Now, the following properties are straightforward to verify.

\begin{lemma}\label{lem: basic properties pi_j}
    Let $g,h\in G$ and $n\in\N$. Then
    \begin{enumerate}
        \item $\psi_n(g) = \pi_1(g)\ldots\pi_n(g)$.
        \item  $\psi_n(g) = \psi_n(h) \iff \forall j\in \set{1,\ldots,n}: \pi_j(g) = \pi_j(h)$.
        \item If $\gamma\in \Gamma_n$, then $\pi_n(\gamma g) = \pi_n(g\gamma) = \pi_n(g)$.
        \item $\pi_j(g) = \pi_j(\psi_n(g))$ for all $j\leq n$.
        \item $\psi_n(gh) = \psi_n(\psi_n(g)\psi_n(h))$ and $\psi_n(g^{-1}) = \psi_n(\psi_n(g)^{-1})$.
    \end{enumerate}
\end{lemma}

Next, we establish a way to compute $\pi_n(gh)$ from $\pi_j(g)$ and $\pi_j(h)$ ($j\in\set{1,\ldots,n}$).
This is again analogous to the addition in the $p$-adic integers. However, as already mentioned above, due to the possible non-commutativity of $G$, the corresponding `carry-over rule' is slightly more involved than in the $p$-adic case. For $h\in G$ we define an automorphism $\alpha_h: G\to G, \: g\mapsto g^{-1}hg$.

\begin{lemma}[Carry over rule]\label{lem: carry over rule}
Let $g,h \in G$ and set $d_1 = 1_G$. 
For $j\in \N$, inductively define $c_j = d_j\cdot \alpha_{\psi_{j-1}(h)}(\pi_j(g)) \cdot \pi_j(h) $ and $d_{j+1} = \varphi_j(c_j) \in \Gamma_j$.
Then, for any $n\in\N$ we have  
\begin{equation} \label{eq: product expansion formula}
gh = \left(\prod_{j=1}^n \psi_j(c_j)\right)\cdot d_{n+1}\cdot \alpha_{\psi_n(h)}(\varphi_n(g))\cdot \varphi_n(h)
\end{equation} 
In particular, $\pi_j(gh) = \psi_j(c_j)$ for all $j\in\N$.

\end{lemma}
The quantities $d_{j}\in $ play the role of the carry over.
\begin{proof}
We proceed by induction on $n$. 
For $n=1$, Lemma~\ref{lem: basic properties pi_j}(1) yields that $\pi_1(g) = \psi_1(g)$ and $\pi_1(h) = \psi_1(h)$.
Since $\alpha_{\psi_0(g)} = \alpha_{1_G}=\mathrm{Id}_G$ and $d_1=1_G$, we thus obtain $c_1 = \psi_1(g)\psi_1(h)$ and
\begin{align*}
    gh &= \psi_1(g)\varphi_1(g) \cdot \psi_1(h) \varphi_1(h)= \psi_1(g)\psi_1(h)\cdot \alpha_{\psi_1(h)}(\varphi_1(g)) \cdot \varphi_1(h)\\
    &=  \psi_1(\psi_1(g)\psi_1(h))\cdot \varphi_1(\psi_1(g)\psi_1(h))\cdot \alpha_{\psi_1(h)}(\varphi_1(g))\cdot \varphi_1(h) \\
    &= \psi_1(c_1)\varphi_1(c_1)\cdot \alpha_{\psi_1(h)}(\varphi_1(g)) \cdot\varphi_1(h).\\
    &= \psi_1(c_1)\cdot d_2\cdot \alpha_{\psi_1(h)}(\varphi_1(g)) \cdot\varphi_1(h).
\end{align*}
Now suppose the induction hypothesis holds for $n\in \N$. Due to Lemma~\ref{lem: basic properties pi_j}(4), we have that $\pi_j(\psi_n(g))=\pi_j(g)$ and $\pi_j(\psi_n(h))=\pi_j(h)$ for all $j=1,\ldots,n$. Since the quantities $c_1,\ldots , c_n$ and $d_1,\ldots , d_{n+1}$ only depend on the first $n$ expansion coefficients, the inductive hypothesis applied to $\psi_n(g)$ and $\psi_n(h)$ implies 
\begin{align} \label{eq: expansion inductive hypothesis application} \psi_n(g)\psi_n(h) &= \left(\prod_{j=1}^n \psi_ j(c_j)\right)\cdot d_{n+1}\cdot \nonumber\underbrace{\alpha_{\psi_n(\psi_n(h))}(\varphi_n(\psi_n(g))\cdot \varphi_n(\psi_n(h))}_{=1_G,\textrm{ since } \varphi_n\circ\psi_n\equiv1_G} \\ & = \left(\prod_{j=1}^n \psi_ j(c_j)\right)\cdot d_{n+1}.
\end{align}
Hence, we obtain 
\begin{align*}
    gh &= \psi_n(g)\pi_{n+1}(g) \varphi_{n+1}(g) \cdot  \psi_n(h)\pi_{n+1}(h)\varphi_{n+1}(h)\\
       &= \psi_n(g)\psi_n(h)\cdot \alpha_{\psi_n(h)}(\pi_{n+1}(g))\cdot \psi_n(h)^{-1}\varphi_{n+1}(g)\psi_n(h)\pi_{n+1}(h)\varphi_{n+1}(h) \\
       &=\psi_n(g)\psi_n(h)\cdot \alpha_{\psi_n(h)}(\pi_{n+1}(g))\\
        &\hspace{2cm}\cdot \pi_{n+1}(h)\underbrace{\pi_{n+1}(h)^{-1}\psi_n(h)^{-1}}_{=\psi_{n+1}(h)^{-1}}\varphi_{n+1}(g)\underbrace{\psi_n(h)\pi_{n+1}(h)}_{=\psi_{n+1}(h)}\varphi_{n+1}(h) \\
       &= \psi_n(g)\psi_n(h)\cdot \alpha_{\psi_n(h)}(\pi_{n+1}(g))\cdot \pi_{n+1}(h)\cdot\alpha_{\psi_{n+1}(h)}(\varphi_{n+1}(g))\cdot \varphi_{n+1}(h)   \\
    &\stackrel{(\ref{eq: expansion inductive hypothesis application})}{=} \left(\prod_{j=1}^n \psi_j(c_j)\right)\cdot \underbrace{d_{n+1} \cdot \alpha_{\psi_n(h)}(\pi_{n+1}(g)) \cdot \pi_{n+1}(h) }_{=\ c_{n+1}\ = \ \psi_{n+1}(c_{n+1})d_{n+2}}\cdot  \alpha_{\psi_{n+1}(h)}(\varphi_{n+1}(g))\cdot \varphi_{n+1}(h)   \\
    &= \left(\prod_{j=1}^{n+1}\psi_j(c_j)\right)   \cdot d_{n+2} \cdot \alpha_{\psi_{n+1}(h)}(\varphi_{n+1}(g))\cdot \varphi_{n+1}(h).
\end{align*}
This completes the induction. The fact that $\pi_j(gh) = \psi_j(c_j)$ now follows from the uniqueness of the expansion in Lemma~\ref{eq: uniqueness of expansion}. 
\end{proof}

\begin{remark}
    Note that if $n>\max\{N(g)+1,N(h)+1\}$, then $\varphi_n(g)=\varphi_n(h)=1_G$. Hence, (\ref{eq: product expansion formula}) simplifies to $gh = \left(\prod_{j=1}^n \psi_j(c_j)\right)\cdot d_{n+1}$. As $d_{n+1}\in\Gamma_{n}$, it has an expansion of the form $\prod_{j=n+1}^{N(d_{n+1})+1} \pi_j(d_{n+1})$. Together, this yields the finite expansion
    $$
    gh = \left(\prod_{j=1}^n \psi_j(c_j)\right)\cdot \left(\prod_{j=n+1}^{N(d_{n+1})+1} \pi_j(d_{n+1})\right)
    $$
    for the product of $g$ and $h$. In particular, we have $N(gh)=N(d_{n+1})$. 
\end{remark}

\begin{remark}\label{rem: carry-over maps} For later use, the following observation is important. 
From the recursive definition in Lemma~\ref{lem: carry over rule}, we obtain maps
$$ d_j: G\times G\to G \quad , \quad (g,h)\mapsto d_j(g,h),$$
 which correspond to the $j$-th carry over that occurs in the computation of the expansion of $gh$.  These satisfy $d_1(g,h) = 1_G$ and
$$
d_{j+1}(g,h)=\varphi_{j-1}\left(d_{j}(g,h)\cdot \alpha_{\psi_{j-1}(h)}(\pi_j(g)) \cdot \pi_j(h)\right).
$$
It follows by induction that each function $d_j$ has finite range, since the ranges of $\pi_j:G\to D_j\cap \Gamma_{j-1}$ and $\psi_j:G\to D_j$ are all finite as well. More precisely, 
we have that $K_1=\{1_G\}$ and 
$$ K_{j+1}\subseteq \varphi_{j-1}(K_j)\cdot D_j^{-1}\cdot(D_j\cap \Gamma_{j-1})\cdot D_j\cdot (D_j\cap \Gamma_{j-1})$$
only depends on the sets $D_1,\ldots,D_{j-1}$,
for all $j>1$. 
\end{remark}

\begin{remark}\label{rem: extension of psi_j and pi_j}
    The maps $\psi_j: G\to D_j$ and $\pi_j:G\to D_j\cap\Gamma_{j-1}$ can be extended to $\overleftarrow{G}$ by defining for $\xi=(\xi_n\Gamma_n)_{n\in\N}$
    $$ \psi_j(\xi) = \psi_j(\xi_j)\:\text{ and }\: \pi_j(\xi) = \pi_j(\xi_j).$$
    This is well defined, because $\psi_j$ and $\pi_j$ are constant on each cylinder of level $j$  (see Lemma~\ref{lem: basic properties pi_j}). In addition, one has that $\psi_j(\xi) = \psi_j(\xi_n)$ as well as $\pi_j(\xi) = \pi_j(\xi_n)$ for all $j\leq n$.
    Moreover, the mapping $\overleftarrow{G}\to \prod_{n\in\N} (D_n\cap\Gamma_{n-1}),\: \xi\mapsto(\pi_n(\xi))_{n\in\N}$ is a bijection.
 \end{remark}

\section{Toeplitz subshifts}\label{sec: Toeplitz}


\subsection{Definitions.}
Let $\Sigma$ be a finite set with at least two elements. We consider $\Sigma$ endowed with the discrete topology and $\Sigma^G$ with the product topology.
The group $G$ acts on $\Sigma^G$ from the left by the shift action:
\begin{align*}
    (gx)(h)=x(hg), \mbox{ for every }g,h\in G, x\in \Sigma^G.
\end{align*}

The dynamical system $(\Sigma^G, G)$ is referred to as the \textbf{full shift} or the full $G$-shift.

Given a closed, $G$-invariant subset $X \subseteq \Sigma^G$, the dynamical system $(X, G)$, obtained by restricting the shift action to $X$, will be called a \textbf{subshift} or a  $G$-subshift.

We say that $x\in\Sigma^G$ is a \textbf{Toeplitz array} if for every $g\in G$ there exists  a finite index subgroup $\Gamma_g$ of $G$ such that $x(g)=x(g\gamma)$ for each $\gamma\in \Gamma_g$.
The subshift $X=\overline{O_G(x)}$ induced by a Toeplitz array $x\in\Sigma^G$ is called a \textbf{Toeplitz subshift}. 
We recall that Toeplitz subshifts are always minimal systems  (see \cite{CoPe08}).
It is well-known that any Toeplitz subshift is an almost one-to-one extension of a suitable $G$-odometer.
With the following notions, one can determine this extension precisely.\\
Let $x\in \Sigma^G$ and $\Gamma$ a finite index subgroup of $G$ and $\alpha\in\Sigma$.
We define
\begin{align*}
    \Per(x,\Gamma,\alpha)&=\{g\in G: x(g\gamma)=\alpha, \mbox{ for every }\gamma\in \Gamma\},\\
    \Per(x,\Gamma)&=\bigcup_{\alpha\in \Sigma}\Per(x,\Gamma,\alpha).
\end{align*}
We say that $\Gamma\leq G$ is a \textbf{period} for $x$ if $\Per(x,\Gamma)\neq \emptyset$.
Using these sets, it is possible to characterize the Toeplitz arrays as the elements $x\in\Sigma^G$ such that there exists a decreasing sequence $(\Gamma_n)_{n\in\mathbb{N}}$ of finite index subgroups of $G$ such that $G=\bigcup_{n\in\mathbb{N}}\Per(x,\Gamma_n)$.
Assume that $x\in\Sigma^G$ is a Toeplitz array.
We say that a period $\Gamma$ of $G$ is \textbf{essential} if we have that, for any $g\in G$, $\Per(x,\Gamma,\alpha)\subseteq \Per(gx,\Gamma,\alpha)$ for each $\alpha\in\Sigma$ implies $g\in\Gamma$.
A decreasing sequence $(\Gamma_n)_{n\in\mathbb{N}}$ of finite index subgroups of $G$ is called a \textbf{period structure} of $x$, if for every $n\in\mathbb{N}$ the group $\Gamma_n$ is an essential period for $x$, and if $G=\bigcup_{n\in\mathbb{N}}\Per(x,\Gamma_n)$.
By \cite[Corollary 6]{CoPe08} we can always guarantee the existence of a period structure for a Toeplitz array.
Let $x\in\Sigma^G$ be a Toeplitz array with period structure given by $(\Gamma_n)_{n\in\mathbb{N}}$. For every $n\in\mathbb{N}$, consider the set 
\begin{align*}
    C_n=\{y\in \overline{O_G(x)}:\Per(y,\Gamma_n,\alpha)=\Per(x,\Gamma_n,\alpha), \mbox{ for every }\alpha\in\Sigma\}.
\end{align*}
The previous sets have the properties that they are clopen and for every $g\in G$, $gx\in C_n$ if and only if $g\in \Gamma_n$.
 Therefore, we can guarantee that for each $n\in\mathbb{N}$, the collection $\{gC_n:g\in D_n\}$ is a clopen partition of $\overline{O_G(x)}$, where $D_n$ is a fundamental domain of $G/\Gamma_n$, i.e., $D_n\subseteq G$ contains exactly one representative of each (left) coset in $G/\Gamma_n$.

\begin{proposition}[\cite{CoPe08}, \cite{Kr10}]\label{MEF_Toeplitz}
    Let $x\in\Sigma^G$ be a Toeplitz array and let $(\Gamma_n)_{n\in\mathbb{N}}$ be a period structure of $x$ and $\overleftarrow{G}$ the $G$-odometer associated to $(\Gamma_n)_{n\in\mathbb{N}}$. 
    The map $\pi:\overline{O_G(x)}\to\overleftarrow{G}$ defined by 
    \begin{align*}
        \pi(y)=(g_n\Gamma_n)_{n\in\mathbb{N}}, \mbox{ where }y\in g_nC_n, \mbox{ for each }n\in\mathbb{N}
    \end{align*}
    is almost one-to-one. The almost automorphic points with respect to this function are given by the Toeplitz elements in $\overline{O_G(x)}$. In other words,
    \begin{align}\label{Sec2:Toeplitzpoints}
        \mathcal{T} := \left\{z\in\overline{O_G(x)}:z \mbox{ is a Toeplitz array}\right\}=\pi^{-1}\left(\left\{\xi\in\overleftarrow{G}: \#\pi^{-1}(\{\xi\})=1\right\}\right).
    \end{align}
    Moreover, $\overleftarrow{G}$ is the maximal equicontinuous factor of $\overline{O_G(x)}$. 
\end{proposition}

\subsection{Regular Toeplitz subshifts}\label{sec: regularity}

 In this subsection, we are interested in determining how close a Toeplitz subshift is --- in a measure-theoretic sense --- to its associated $G$-odometer. To that end, let $x\in\Sigma^G$ be a Toeplitz array with period structure $(\Gamma_n)_{n\in\mathbb{N}}$.
Further, let $(n_i)_{i\in\N}$ and $(D_i)_{i\in\mathbb{N}}$ be as in Lemma~\ref{lem: ex. of seq. of fundamental domains}. By taking a subsequence of $(\Gamma_n)_{n\in\mathbb{N}}$ if necessary, we may assume that $n_i=i$ for each $i\in\mathbb{N}$. 
Note here that the $G$-odometer that is obtained from a subsequence of $(\Gamma_n)_{n\in\N}$ is always conjugate to the odometer obtained from the sequence $(\Gamma_n)_{n\in\N}$ itself (see \cite[Lemma 2]{CoPe08}).
Let $\overleftarrow{G}$ be the $G$-odometer associated to $(\Gamma_n)_{n\in\mathbb{N}}$. By Proposition~\ref{MEF_Toeplitz}, the odometer $(\overleftarrow{G},G)$ is the maximal equicontinuous factor of $(\overline{O_G(x)},G)$. 

Consider the sequence $(d_n)_{n\in\mathbb{N}}$, given by
\begin{align*}
    d_n=\dfrac{\#(D_n\cap \Per(x,\Gamma_n))}{\#D_n}.
\end{align*}
This sequence is increasing and converges to some $d\in(0,1]$ (see, for example, \cite{CeCoGo23}).
Consider the set of Toeplitz arrays $\mathcal{T}$ defined in \eqref{Sec2:Toeplitzpoints}.

\begin{proposition}[{\cite[Proposition 3.3]{CeCoGo23}}] \label{prop: regular Toeplitz flows}
    The following statements are equivalent
    \begin{enumerate}
        \item $\nu(\pi(\mathcal{T}))=1$.
        \item There exists an invariant probability measure $\mu$ on $\overline{O_G(x)}$ such that \mbox{$\mu(\mathcal{T})=1$.}
        \item There exists a unique invariant probability measure $\mu$ on $\overline{O_G(x)}$ such that $\mu(\mathcal{T})=1$.
        \item $d=1$.
    \end{enumerate}
\end{proposition}
If the Toeplitz array $x\in\Sigma^G$ satisfies the previous statements, we say that it is a \textbf{regular Toeplitz array} and call $\overline{O_G(x)}$ a \textbf{regular Toeplitz subshift}.
If $x\in\Sigma^G$ does not satisfy the previous statements, we say that it is \textbf{irregular}.
See \cite{FK18} and \cite{LaStra18}  for more details about regularity.

\subsection{Semicocycles} \label{sec: semicocycles}  

The following definitions and some results are taken from \cite{CoPe08}  (a more general setting can be found in \cite{FK18}),  and are natural generalizations of the case $G=\Z$ (c.f.\ \cite{Do05}).
Suppose that $G$ is a residually finite group with a sequence $(\Gamma_n)_{n\in\N}$ of finite index subgroups with trivial intersection.
Furthermore, let $\overleftarrow{G} = \varprojlim(G/\Gamma_n,\phi_n)$ be the associated $G$-odometer and $\tau: G \to \overleftarrow{G}$ be the canonical embedding.
By identifying $G$ with $\tau(G)$, the metric $d$ on $\overleftarrow{G}$ (as defined in \eqref{eq: metric odometer}) induces a metric on $G$, which we shall also denote by $d$.
Let $K$ be a compact metric space.
We call a function $\varphi: G \to K$ a {\bf semicocycle} on $\overleftarrow{G}$ if $\varphi$ is continuous with respect to this induced metric.
We naturally equip the set $K^G$ of all functions $\varphi: G \to K$ with the product topology and with the shift $G$-action $(g\varphi)(h) = \varphi(hg)$, for every $g,h\in G$, in order to obtain a dynamical system $(K^G,G)$. 
Recall that a point $x$ in a topological dynamical system $(X,G)$ is called \textbf{regularly recurrent} if, for any open neighbourhood $U$ of $x$, the set $\{g\in G\mid gx\in U\}$ contains a cocompact (in our setting finite index) subgroup of $G$. 

\begin{proposition}\label{semicocycles}
Let $\varphi\in K^G$ be non-periodic. The following statements are equivalent:
\begin{enumerate}   
  \item  There exists an odometer $\overleftarrow{G}$ defined from a nested sequence of finite index subgroups of $G$ with trivial intersection  such that $\varphi$ is a semicocycle on $\overleftarrow{G}$.
  \item $\varphi$ is a regularly recurrent point of $(K^G,G)$.
  \end{enumerate}
\end{proposition}
\begin{proof}
 Let $F\subseteq G$ be a finite set and $V_g\subseteq K$ be an open neighborhood of $\varphi(g)$, for every $g\in F$. The set
$V=\{f\in K^G\mid  f(g) \in V_g, \forall g \in F\}$ is an open neighborhood of $\varphi$ in $K^G$. Moreover, observe that sets of this form constitute a neighborhood basis of $\varphi$. 

Suppose that $\varphi$ is a semicocycle on an odometer $\overleftarrow{G}$, defined from the sequence $(\Gamma_n)_{n\in\mathbb{N}}$  of finite index subgroups of $G$  with trivial intersection. 
The continuity of $\varphi$ implies that there exists $k\in\mathbb{N}$ such that if $h\in G$ verifies $d(\tau(h),\tau(g))\leq \frac{1}{2^k}$, then $\varphi(h)\in V_g$, for every $g\in F$.  Since $d(\tau(g\gamma),\tau(g))\leq \frac{1}{2^k}$, for every $\gamma \in \Gamma_k$, then $\varphi(g\gamma)=(\gamma\varphi)(g)\in V_g$, for every $\gamma\in \Gamma_k$ and $g\in F$.   We have that $\gamma\varphi \in V$, for every $\gamma\in \Gamma_k$. This shows that $\varphi$ is regularly recurrent.   

Suppose that $\varphi$ is a regularly recurrent point of $\Sigma^G$. Since $G$ is countable, there exists an increasing sequence $(F_n)_{n\in\mathbb{N}}$ of finite subsets of $G$ such that $G=\bigcup_{n\in\mathbb{N}}F_n$. Let $(\varepsilon_n)_{n\in\mathbb{N}}$ be a decreasing sequence of positive numbers converging to $0$. Let $V_n$ be the neighborhood $V$ for  $F=F_n$ and $V_g=B_{\varepsilon_n}(\varphi(g))$, for every $g\in F_n$ and  $n\in\mathbb{N}$. The regular recurrence of $\varphi$ implies that for every $n\in\mathbb{N}$, there exists a finite index subgroup $H_n$, such that $\gamma\varphi\in V_n$, for every $\gamma\in H_n$. Since $\bigcap_{n\in\mathbb{N}}V_n=\{\varphi\}$ and $\varphi$ is non-periodic, we have $\bigcap_{n\in\mathbb{N}}H_n=\{1_G\}$. Let us define $\Gamma_1=H_1$ and $\Gamma_{n+1}=\Gamma_n\cap H_{n+1}$, for every $n\in \mathbb{N}$. Observe that $(\Gamma_n)_{n\in\mathbb{N}}$ is a decreasing sequence of finite index  subgroups of $G$ with trivial intersection. Let $\overleftarrow{G}$ be the odometer associated to $(\Gamma_n)_{n\in\mathbb{N}}$.
Let $g\in G$ and $n\in\mathbb{N}$ be such that $g\in F_n$ and let $k_n\geq n$ be such that $\frac{1}{2^{k_n}}<\varepsilon_n$. 
Let $h\in G$ be such that $d(\tau(h),\tau(g))\leq \frac{1}{2^{k_n}}$. We have that $h\Gamma_n =g\Gamma_n$, which implies  that $h=g\gamma$ for some $\gamma\in \Gamma_n$. Then $\varphi(h)=\varphi( g\gamma)=(\gamma\varphi)(g)\in B_{\varepsilon_n}(\varphi(g))$. This shows that $\varphi$ is continuous in $g\in G$ and hence a semicocycle on $\overleftarrow{G}$.
\end{proof}

Since a semicocycle $\varphi$ is essentially a continuous map which is defined on a dense subset of $\overleftarrow{G}$, it is natural to attempt to extend $\varphi$ continuously to some larger subset of $\overleftarrow{G}$. 
For this we define
$$ \Phi = \overline{\{(g,\varphi(g)):g\in G\}} \tm \overleftarrow{G} \times K,$$
where the closure is taken in the product topology.
Furthermore, given $\xi \in \overleftarrow{G}$, we define $\Phi(\xi) = \{k\in K:(\xi,k)\in \Phi\}$.
Lastly, we set $C_\varphi= \{\xi\in \overleftarrow{G}:\# \Phi(\xi) = 1\}$ and $D_\varphi = \overleftarrow{G}\setminus C_\varphi$. 
By continuity of $\varphi$ it follows for $g\in G$ that $\Phi(g) = \{\varphi(g)\}$, so in particular $g\in C_\varphi$.
Thus, we can extend $\varphi$ to $C_\varphi$ by letting $\varphi(\xi)$ be the unique element of $\Phi(\xi)$.
Since we have $G\subseteq C_\varphi$, the set $C_\varphi$ is always a dense $G_\delta$ subset of $\overleftarrow{G}$. 
In the case that $K$ is finite, $C_\varphi$ is even open.
We call the semicocycle $\varphi$ {\bf invariant under no rotation} if 
$\Phi(g\xi_1)=\Phi(g\xi_2)$ for every $g\in G$ implies $\xi_1=\xi_2$. 

\begin{proposition}\label{semicocycles_toeplitz}
Let $\Sigma$ be  a finite alphabet and $x\in \Sigma^G$ be a non-periodic element. Then the following statements are equivalent:
\begin{enumerate}
\item $x$ is a   Toeplitz array.

\item $x$ is a regularly recurrent point of the full shift $(\Sigma^G,G)$.

\item The map $x:G\to \Sigma$  is a semicocycle on the odometer $\overleftarrow{G}$, where $(\overleftarrow{G},G)$ is the maximal equicontinuous factor of $(\overline{O_G(x)},G)$.
\end{enumerate}
\end{proposition}

\begin{proof}
The equivalence between (1) and (2) is consequence of \cite[Proposition 5 and Theorem 2]{CoPe08} and Proposition~\ref{semicocycles}. The fact that (3) implies (2) follows from Proposition~\ref{semicocycles}. Hence, it suffices to show (1) $\Rightarrow$ (3).

In order to do so, suppose that $x\in\Sigma^G$ is a Toeplitz array. Let $(\Gamma_n)_{n\in\mathbb{N}}$ be a period structure of $x$. Since $x$ is non-periodic, this is a decreasing sequence of finite index  subgroups with trivial intersection. Let $\overleftarrow{G}$ be the odometer associated to $(\Gamma_n)_{n\in\mathbb{N}}$. Proposition~\ref{MEF_Toeplitz} implies  that $(\overleftarrow{G},G)$  is the maximal equicontinuous factor of $(\overline{O_G(x)},G)$. Let $g\in G$ and $n\in\mathbb{N}$ be such that $g\in\Per(x,\Gamma_n)$. Let $h\in G$ be such that $d(\tau(h),\tau(g))\leq \frac{1}{2^n}$. 
This implies there exists $\gamma\in \Gamma_n$ such that $h= g \gamma$. Since $g\in \Per(x,\Gamma_n)$, we have $x(h)=x(g\gamma)=x(g)$, which means that $x:G\to \Sigma^{G}$ is continuous on $G$, with respect the topology induced by $\overleftarrow{G}$. This shows that $x:\Sigma\to G$ is a semicocycle on $\overleftarrow{G}$.
   \end{proof}

Let $x\in\Sigma^G$ be a non-periodic Toeplitz array. According to Proposition~\ref{semicocycles_toeplitz}, the map $g\in G\to x(g)\in\Sigma$ is a semicocycle on $\overleftarrow{G}$, with $(\overleftarrow{G},G)$ the maximal equicontinuous factor of $(\overline{O_G(x)},G)$. Proposition~\ref{MEF_Toeplitz} implies there exists an almost one-to-one factor map  $\pi:\overline{O_G(x)}\to \overleftarrow{G}$ such that $\pi^{-1}\left(\{(1_G\Gamma_n)_{n\in\mathbb{N}}\}\right)=\{x\}$.\\
Let $\Phi=\overline{\{(g,x(g)):g\in G\}}$, and $\Phi(\xi) = \{a\in \Sigma:(\xi,a)\in \Phi\}$, for every $\xi\in \overleftarrow{G}$.   
For $\xi=(z_n\Gamma_n)_{n\in\mathbb{N}}\in \overleftarrow{G}$, we denote
$$
\Per(\xi)=\bigcup_{n\in\mathbb{N}}\Per(x,\Gamma_n)z_n^{-1}.
$$

\begin{lemma}\label{Characterization_Phi}
Let $x\in \Sigma^G$ be a non-periodic Toeplitz array, $\overleftarrow{G}$ its maximal equicontinuous factor, and $\pi:X\to \overleftarrow{G}$ be an almost one-to-one factor map such that\\ $\pi^{-1}\left(\{(1_G\Gamma_n)_{n\in\mathbb{N}}\}\right)=\{x\}$, where $(\Gamma_n)_{n\in\mathbb{N}}$ is a decreasing sequence of finite index subgroups with trivial intersection that defines $\overleftarrow{G}$. 
Then, for every $\xi\in \overleftarrow{G}$, we have 
$$
 \Phi(\xi)=\{y(1_G): y\in \pi^{-1}(\{\xi\})\}. 
$$
Furthermore, if $g\in G$, then $\#\Phi(g\xi)=1$ if and only if $g \in \Per(\xi)$.

\end{lemma}
\begin{proof}
Let $\xi=(z_n\Gamma_n)_{n\in\mathbb{N}}$.
Suppose that $a\in \Phi(\xi)$. Then there exists a sequence $(g_i)_{i\in\mathbb{N}}$ in $G$ such that $g_i\to \xi$ and $x(g_i)=a$, for every (sufficiently large) $i\in \mathbb{N}$. Let $y\in \overline{O_G(x)}$ be an accumulation point of $(g_ix)_{i\in\mathbb{N}}$. We have $y(1_G)=a$ and $\pi(y)=\xi$.  

Conversely, suppose that $y\in \pi^{-1}(\{\xi\})$. Let $(g_ix)_{i\in\mathbb{N}}$ be a sequence that converges to $y$. This implies that $x(g_i)\to y(1_G)$ and $g_i\to \xi$, from which we get $y(1_G)\in \Phi(\xi)$. 
This shows the first statement.

Now, let $g\in G$. Suppose that $g\in \Per(x,\Gamma_n,\alpha)z_n^{-1}$, for some $\alpha\in\Sigma$ and $n\in\mathbb{N}$. If $a\in \Phi(g\xi)$ then there exists $y\in \pi^{-1}(\{g\xi\})$ such that $y(1_G)=a$. On the other hand, if $(g_ix)_{i\in\mathbb{N}}$ is a sequence that converges to $y$, then for every sufficiently large $i$ we have $g_i\in gz_n\Gamma_n$ and $x(g_i)=y(1_G)=a$. Since $g_i\in gz_n\Gamma_n$, there exists $\gamma\in \Gamma_n$ such that $g_i=gz_n\gamma$. This implies $a=x(g_i)=x(gz_n\gamma)=\alpha$, because $gz_n\in\Per(x,\Gamma_n,\alpha)$. This shows that $\#\Phi(g\xi)=1$. 

Conversely, if $g\notin \Per(x,\Gamma_n)z_n^{-1}$ for every $n\in \mathbb{N}$, then for every $n\in \mathbb{N}$ there exists $\gamma_n\in\Gamma_n$ such that $x(gz_n\gamma_n)\neq x(gz_n)$. Let $y_1$ and $y_2$ be limits  of  simultaneously convergent subsequences $((gz_{n_i})x)_{i\in\mathbb{N}}$ and $((gz_{n_i}\gamma_{n_i})x)_{n\in\mathbb{N}}$, respectively. We have that $y_1,y_2\in \pi^{-1}(\{g\xi\})$, and for every sufficiently large $i$, $y_1(1_G)=x(gz_{n_i})\neq x(gz_{n_i}\gamma_{n_i})=y_2(1_G)$, which implies that $\Phi(g\xi)$ has at least two different elements.
\end{proof}
 
\begin{remark}\label{rem: Toeplitz is invariant under no rotations}
It is possible to show that $\Phi$ of Lemma~\ref{Characterization_Phi} satisfies $\Phi(g\xi_1)=\Phi(g\xi_2)$ for every $g\in G$ implies $\xi_1=\xi_2$, meaning that $x$ as a semicocycle on $\overleftarrow{G}$ is invariant under no rotation (as expected from \cite{FK18}). 
\end{remark}

\begin{theorem}\label{thm: chara. Toeplitz semicocycle}
 Let $\Sigma$  be a finite alphabet and $x\in\Sigma^G$ be a non-periodic element. Let $\overleftarrow{G}$ be a free odometer with group structure. 
 The following statements are equivalent:
 \begin{enumerate}
\item   $x$ is  Toeplitz  and $(\overleftarrow{G},G)$ is the maximal equicontinuous factor of $(\overline{O_G(x)},G)$.

\item The map $G\to \Sigma,\:g\mapsto x(g)$ is a semicocycle on $\overleftarrow{G}$   invariant under no rotation. 
\end{enumerate}
Moreover, in this case we have $$\nu(\{\xi\in\overleftarrow{G}: \#\Phi(\xi)=1  \}) = d,$$ where $\nu$ is the normalized Haar measure of $\overleftarrow{G}$,  $\Phi$ is defined as above and $d= \lim_{n\to\infty} d_n$ is the regularity constant of $x$ defined in Section~\ref{sec: regularity}.
\end{theorem}
\begin{proof}
We have (1) implies  (2) due to Proposition~\ref{semicocycles_toeplitz} and Remark~\ref{rem: Toeplitz is invariant under no rotations}. Conversely, if $x:G\to \Sigma$ is a semicocycle on $\overleftarrow{G}$, then Proposition~\ref{semicocycles_toeplitz} implies that $x$ is Toeplitz. If in addition, $x$ is invariant under no rotation,   by \cite[Theorem 2.5]{FK18} or \cite[Corollary 5]{CoPe08} we have that $(\overleftarrow{G},G)$ is the maximal equicontinuous factor of $(\overline{O_G(x)},G)$.  
       
  Let $(\Gamma_n)_{n\in\mathbb{N}}$ be a decreasing sequence of normal finite index subgroups with trivial intersection that defines $\overleftarrow{G}$.  By Lemma~\ref{Characterization_Phi}, we have
  $$
  C_x=\{\xi\in\overleftarrow{G}: \#\Phi(\xi)=1  \}=\{\xi\in\overleftarrow{G}: 1_G\in \Per(\xi)\}.
  $$
  If $\xi=(z_n\Gamma_n)_{n\in\mathbb{N}}$ then $1_G\in \Per(\xi)$ if and only if there exists $n\in\mathbb{N}$ such that $z_n \in\Per(x,\Gamma_n)$. This implies that
  $$
  C_x=\bigcup_{n\in\mathbb{N}}\bigcup_{\gamma\in \Per(x,\Gamma_n)\cap D_n}\{(z_k\Gamma_k)_{k\in\mathbb{N}}\in\overleftarrow{G}: z_n\in \gamma\Gamma_n\},
  $$
which yields
    $$ \nu(C_x) = \lim_{n\to\infty} \frac{\#(\Per(x,\Gamma_n)\cap D_n)}{\#D_n} = d.$$
    \end{proof}

\section{Cut and project schemes and almost automorphic systems} \label{sect: CPS}

Let $G$ be a countable discrete group. The space of subsets of $G$ can be identified with $\{0,1\}^G$ through the transformation which maps $S\tm G$ to its indicator function.
We have seen in Section~\ref{sec: Toeplitz} that we can endow $\{0,1\}^G$ with the product topology and a shift action to obtain a topological dynamical system $(\{0,1\}^G,G)$, called the $0$-$1$ full $G$-shift.

The aim of this section is to show that the almost automorphic elements of $\{0,1\}^G$ can  be obtained as a model set via a suitably defined cut and project scheme (CPS). We provide the definition of these notions in the setting of $G$ being countable and discrete.\\

Let $H$ be a locally compact, second countable group. Further, let $\pi_1$ denote the projection $G\times H \to G$ to the first coordinate and $\pi_2:G\times H \to H$ the projection to the second coordinate. 
Suppose that $\mathcal{L}\subseteq G\times H$ is a lattice, that is, a discrete cocompact subgroup --- 
cocompactness meaning that the quotient space $(G\times H)/\mathcal{L}$ is compact with respect to the quotient topology. 
Since $G\times H$ is locally compact, it can be seen that this is equivalent to $\mathcal{L}$ being syndetic.
If additionally the restriction $(\pi_1)\vert_\mathcal{L}$ is injective and $\pi_2(\mathcal{L})$ is dense in $H$, then we call the triple $(G,H,\mathcal{L})$ a \textbf{cut and project scheme} (CPS).
Let $L = \pi_1(\mathcal{L})$.
In this situation, the so-called \textbf{star-map} $$\cdot^\star: L\to H\quad , \quad \: g \mapsto g^*=\pi_2\circ (\pi_1\vert_\mathcal{L})^{-1}(g)$$ is well-defined.  
Given a subset $W\subseteq H$, which we call a \textbf{window}, the corresponding \textbf{model set} (or cut and project set) $\Lambda(W)$ is given by
$$ \Lambda(W) = \{g \in L:g^\star \in W\}. $$


We call a window $W\subseteq H$ and its associated model set \textbf{proper} if $W$ is compact with $W=\ol{\mathrm{int}(W)}$.
Moreover, we call $W$ {\bf irredundant} if $Wh = W$ implies $h=1_H$, {\bf generic} if $\partial W \cap L^\star = \emptyset$, {\bf regular} if $\mu_H(\partial W)=0$ and {\bf irregular} if $\mu_H(\partial W)>0$, where $\mu_H$ denotes the (left) Haar measure on $H$. \\[5pt]
By definition, we can associate to the CPS $(G,H,\mathcal{L})$ a compact space
$$ \T := (G\times H)/ \mathcal{L},$$
which is called the {\bf torus} of the CPS and it is equipped with a natural $G$-action that turns $(\T,G)$ into a minimal dynamical system.

It turns out that this action is a factor of the dynamical system given by the closure of the orbit of the model set, a fact which is commonly referred to as {\em `torus parametrization'}. A proof is given in \cite{BaLeMo07}  for the case of locally compact abelian groups, but extends to the non-commutative setting with modest modifications (see \cite{BjHaPo18}, or \cite{DJL} for an overview). 
In the next section we give the details for the case of cut and project schemes $(G,H,\mathcal{L})$ such that $(H,\tau)$ is a metric compactification of $G$ and $\mathcal{L}=\{(g,\tau(g): g\in G\}$.
In this situation, one can see that the dynamical system $(\T,G)$ is in fact conjugate to $(H,G)$.

\subsection{Almost automorphic systems and model sets}\label{sec: almost auto as model set} In \cite{BaJäLe16}, it is shown that all Toeplitz sequences correspond to suitably defined model sets in $\mathbb{Z}$. The aim of this section is to extend that result to the setting of almost automorphic arrays over metrizable compactifications of countable discrete groups, leading, as a corollary, to a characterization of $0$-$1$-Toeplitz arrays as model sets within residually finite groups.
 
\medskip

If $(H,\tau)$ is a metric compactification of $G$, then $(G,H, \mathcal{L})$ is a cut and project scheme, where $\mathcal{L}=\{(g,\tau(g)): g\in G\}$. In this case, the star-map is simply the mapping $g\mapsto \tau(g)$ and the model set $\Lambda(W)$ is given by
$$
\Lambda(W)=\{g\in G: \tau(g)\in W\}.
$$
We identify the set $\Lambda(W)$ with the element $x_W\in \{0,1\}^G$ defined by
\begin{equation}\label{eq: def of model set array}
x_W(g)=1 \iff \tau(g)\in W.
\end{equation}
Hence, we will also refer to the array $x_W$ as a model set.
Remember that $W$ is proper if $W$ is compact with $W=\overline{\mathrm{int}(W)}$, and 
 $W$ is {\bf irredundant} if $W\xi = W$ implies $\xi=1_{H}$. In this setting,  $W$ is {\bf generic} if $\partial W \cap \tau(G) = \emptyset$.

Let us define $$\mathcal{W}(H)=\{W\subseteq H: W \mbox{ is proper, irreduntant and generic}\}. $$
The goal of this section is to show that there is a one-to-one correspondence between $\mathcal{W}(H)$ and the elements $x\in\{0,1\}^G$ that are almost automorphic on $H$. First, note that denseness of $\tau(G)$ in $\overleftarrow{G}$ easily implies 

\begin{lemma}\label{interior}
If $W=\overline{\mathrm{int}(W)}$, then 
$$
 \overline{\tau(G)\cap \mathrm{int}(W)}=\overline{\tau(G)\cap W}=W.
$$
\end{lemma}

If $G$ is abelian, the following proposition is a special case of a more general result (for locally compact abelian groups) that can be found in \cite{BaLeMo07, BjHaPo18}. Since we also deal with non-commutative groups, a proof is included for the convenience of the reader (see also \cite{DJL}). 
\begin{proposition}\label{prop: torus para bis 1}
    Let $W\subseteq H$ be a proper and irredundant window. 
    Then there exists a unique factor map $\beta: \overline{O_G(x_W)} \to H$ with $\beta(x_W) = 1_{H}$.
    This map satisfies 
    \begin{equation}\label{eq: chara beta}
    \beta(x) = \xi   \iff \tau(G)\cap \mathrm{int}(W)\xi^{-1}\subseteq \{\tau(g): x(g)=1\}\subseteq \tau(G)\cap W\xi^{-1},
    \end{equation}
    for $x \in \overline{O_G(x_W)}$.
    
If in addition, $W$ is generic, then $\beta^{-1}(\{1_{H}\})=\{x_W\}$, which implies that $x$ is almost automorphic with respect to $H$.
\end{proposition}

\begin{proof}
Observe that $gx_W=x_W$ if and only if $W=W\tau(g)^{-1}$. Then, since $W$ is irredundant,  the array $x_W$ is non-periodic and the map $\beta: O_G(x_W)\to \overleftarrow{G}$ given by $\beta(gx_W)=\tau(g)$ for every $g\in G$, is well defined.

 Now, we will show that $\beta$ admits a unique continuous extension to $\overline{O_G(x_W)}$. For that, let $x\in \overline{O_G(x_W)}$ and $(h_ix_W)_{i\in\mathbb{N}}$ be a sequence that converges to $x$. Then for every $g\in G$, there exists $i_g\in\mathbb{N}$ such that $x_W(gh_i)=x(g)$ for every $i\geq i_g$. Thus, if $x(g)=1$ then $\tau(g)\in W\tau(h_i)^{-1}$, for every $i\geq i_g$. Suppose that $\xi_1$ and $\xi_2$ are two accumulation points of $(\tau(h_i))_{i\in\mathbb{N}}$. Then $\{\tau(g): x(g)=1\}\subseteq \tau(G)\cap W\xi_s^{-1}$, for $s=1,2$. Now, suppose that $g\in G$ is such that  $\tau(g)\in \tau(G)\cap \mathrm{int}(W)\xi_s^{-1}$. Then there exists $w\in \mathrm{int}(W)$  such that $\tau(g)\xi_s=w$. This implies that for a subsequence $(h_{i_j})_{j\in\mathbb{N}}$ we have that for every sufficiently large $j$, the element $\tau(gh_{i_j})$ is in a neighborhood of $w$ contained in $W$. This implies that $x_W(gh_{i_j})=1$ for every sufficiently large $j$, from which we deduce that $x(g)=1$. We have shown
 \begin{equation}\label{eq-1}
 \tau(G)\cap \mathrm{int}(W)\xi_s^{-1}\subseteq \{\tau(g): x(g)=1\}\subseteq \tau(G)\cap W\xi_s^{-1}, \mbox{ for every } s=1,2.
 \end{equation}
 Since $\mathrm{int}(W)\xi_s^{-1}=\mathrm{int}(W\xi_s^{-1})$ (the map $w\to w\xi_s^{-1}$ is a homeomorphism), Lemma~\ref{interior} implies that
$$  
 W\xi_s^{-1}=\overline{\{\tau(g): x(g)=1\}}, \mbox{ for } s=1,2.
$$
 Since $W$ is irredundant, we deduce $\xi_1=\xi_2=:\xi$. 
 We have shown that if $h_ix_W\to x$ then $(\tau(h_i))_{i\in\mathbb{N}}$ converges to some $\xi\in H$. Observe that $\beta$ is uniformly continuous on $O_G(x_W)$. Indeed, if  $\beta$ is not uniformly continuous, then there exists $\varepsilon>0$ such that for every $n\in\mathbb{N}$ there exists $h_n, g_n\in G$ such that $d(h_nx_W,g_nx_W)\leq \frac{1}{n}$ and $d(\tau(h_n),\tau(g_n)) =d(\beta(h_nx_W), \beta(g_nx_W))\geq \varepsilon.$ Taking subsequences, this implies that $h_nx_W, g_nx_W\to x$, where $x$ is some element in $\overline{O_G(x_W)}$. From the previous discussion, it follows that $\tau(h_n)\to \xi_1$ and $\tau(g_n)\to \xi_2$. By hypothesis, we have $\xi_1\neq \xi_2$. Nevertheless, taking $f_{2n}=h_n$ and $f_{2n-1}=g_n$ we have $f_nx_W\to x$, which implies that  $\tau(f_n)\to \xi$, for some $\xi\in H$. Since $\xi_1$ and $\xi_2$ are accumulation points of $(\tau(f_n))_{n\in\mathbb{N}}$, we must have $\xi_1=\xi=\xi_2$, which is a contradiction.  Thus, $\beta$ is uniformly continuous in $O_G(x_W)$, which implies that it extends to a unique continuous function in $\overline{O_G(x_W)}$. It is not difficult to see that this extension is a factor map.   


Furthermore, observe that equation~\eqref{eq-1} implies that if $\beta(x)=\xi$, then
\begin{equation}\label{eq: squeeze for fiber element}
 \tau(G)\cap \mathrm{int}(W)\xi^{-1}\subseteq \{\tau(g): x(g)=1\}\subseteq \tau(G)\cap W\xi^{-1}.
\end{equation}
Conversely, we have seen that due to irredundancy of $W$, at most one $\xi$ can satisfy \eqref{eq: squeeze for fiber element} for fixed $x\in \ol{O_G(x_W)}$.
This shows that $\beta$ is indeed characterized by $\eqref{eq: chara beta}$.

Lastly, if $W$ is generic, then $\tau(G)\cap \mathrm{int}(W)=\tau(G)\cap W$, which implies that $\beta(x)=1_{H}$ if and only if $x=x_W$. Since $\{\beta(gx_W): g\in G\}=\tau(G)$ is dense in $H$, the map $\beta$ is almost one-to-one, and $(\overline{O_G(x_W)},G)$ is almost automorphic. 
\end{proof}

The next result corresponds to Theorem~\ref{thm: chara almost automorphic (introduction)}.
 
 \begin{theorem}\label{thm: almost-automorphic-model-set}
    Let $G$  be a countable discrete group and $(H,\tau)$ be a metrizable compactification of $G$. Then the map
    $$
    \Lambda: \mathcal{W}(H)\to \{x\in \{0,1\}^{G}: x \mbox{ is almost automorphic on } H\}
    $$
    given by
    $$
    \Lambda(W)=x_W, \mbox{ for every } W\in \mathcal{W}(H),
    $$
    is well defined and bijective.  
    
    Moreover, if $\beta: \overline{O_G(x)}\to H$ is the factor map such that $\beta^{-1}(\{1_H\})=\{x\}$, then the window $W$ associated to $x$ is given by
    $$
    W=\{\xi\in H: \mbox{ there exists } y\in \beta^{-1}(\{\xi\}) \mbox{ such that } y(1_G)=1\}.
    $$
    and
    \begin{align}\label{eq: chara boundary}
    \begin{split}
    \partial W&=W\setminus \{\xi\in H: y\in \beta^{-1}(\{\xi\}) \implies y(1_G)=1\}\\
    &= H \setminus \{\xi\in H: y(1_G) = y^\prime(1_G)\text{ for all } y,y^\prime \in \beta^{-1}(\{\xi\})\}.
    \end{split}
    \end{align}
\end{theorem}
\begin{proof}
Proposition~\ref{prop: torus para bis 1} ensures that the map $\Lambda$ is well defined. Furthermore, injectivity of $\Lambda$ follows from the genericity of the windows $W\in\cW(H)$ and the denseness of $\tau(G)$.  

Let $x\in \{0,1\}^G$ be such that there exists an almost one-to-one factor map $\beta:\overline{O_G(x)}\to H$ verifying $\beta^{-1}(\{1_H\})=\{x\}$. for $\xi\in H$, define the set  
$$
\Phi(\xi)=\{y(1_G): y\in \beta^{-1}(\{\xi\})\}.
$$
Let $U_s=\{\xi\in H: \Phi(\xi)=\{s\} \}$, for $s=0,1$. We claim that $U_0$ and $U_1$ are non-empty open sets. Indeed, due to $x$ being non-periodic and $\#\beta^{-1}(\{\beta(gx)\})=1$ for every $g\in G$, 
 the image by $\beta$ of $O_G(x)$ intersects both sets $U_0$ and $U_1$. Moreover, if $\xi\in U_s$  is such that there exists a sequence $(\xi_i)_{i\in\mathbb{N}}$ in $H\setminus U_s$ that converges to $\xi$, then taking $y_i\in\beta^{-1}(\{\xi_i\})$  with $y_i(1_G)\neq s$, we get a subsequence $(y_{i_j})_{j\in\N}$ that converges to some $y^\prime$ that satisfies $y^\prime\in \beta^{-1}(\{\xi\})$ and $y^{\prime}(1_G)\neq s$, which is a contradiction. Hence $U_s$ is open. \\
Let $D=\{\xi\in H: \Phi(\xi)=\{0,1\}\}$ and
 $$
 W=H\setminus U_0=U_1\uplus D=\{\xi\in H: 1\in \Phi(\xi)\}. 
 $$
Since for every $g\in G$ we have $\Phi(\tau(g))=\{x(g)\}$, then
 $$  \{g\in G:\tau(g) \in W\} = \{g\in G:1\in \Phi(\tau(g))\} = \{g\in G:x(g)=1\},$$
which implies $x_W=x$.

 Now we will show that $W$ is a proper, irredundant and generic window. First, observe that $W$ is closed because it is the complement of $U_0$. Since $U_1$ is open, $\partial W\subseteq D$. On the other hand, if $\xi\in D$ and $y\in\beta^{-1}(\{\xi\})$ is such that $y(1_G)=0$, then every sequence $(g_ix)_{i\in\mathbb{N}}$ that converges to $y$ verifies  $\tau(g_i)\to \xi$   and $\Phi(\tau(g_i))=\{0\}$ for every $i$ big enough, which implies that $\xi\in \partial W$. Since $D=\partial W$, we get $U_1=\mathrm{int}(W)$ and $W=\overline{\mathrm{int}(W)}$, showing that $W$ is proper. Furthermore, since $\emptyset=\tau(G)\cap D =\tau(G)\cap\partial W$, it follows that $W$ is generic. In order to prove that $W$ is irredundant, suppose there exists $\xi\in H$ such that $W\xi^{-1}=W$. This  implies  that $1\in \Phi(k\xi)$ if and only if $1\in\Phi(k)$, for every $k\in H$. In particular, $1\in \Phi(g\xi)$ if and only if $\{1\}=\Phi(g1_{H})=\Phi(\tau(g))$, which implies $\{g\in G: y(g)=1\}\subseteq \{g\in G: x(g)=1\}$, for every $y\in  \beta^{-1}(\{\xi\})$.  Now, suppose that $(h_ix)_{i\in\mathbb{N}}$ converges to $y\in\beta^{-1}(\{\xi\})$, and let $g\in G$ be such that $x(g)=1$, that is, $\tau(g)\in W\xi^{-1}=W$. Since $W$ is generic, $\tau(g)\in \mathrm{int(W)}=\mathrm{int}(W)\xi^{-1}$ and then $\tau(g)\xi\in \mathrm{int}(W) $. Observe that $\beta(h_ix)\to \beta(y)=\xi$ implies $\tau(h_i)\to \xi$, and then $\tau(g)\tau(h_i)\to \tau(g)\xi$. From this we get $\tau(gh_i)\in \mathrm{int}(W)$, and then $x(gh_i)=1$ for every sufficiently large $i$. Since $x(gh_i)\to y(g)$,    we get $y(g)=1$. This shows that $\{g\in G: x(g) = 1\} = \{g\in G: \tau(g) \in W\} \tm \{g\in G: y(g)=1\}$, which leads to $y=x$ and $\xi=1_{H}$. Therefore, $W$ is irredundant.  
\end{proof}

Combining the previous two results, we obtain the following corollary.
\begin{corollary}\label{coro-characterization}
  Let $G$ be a discrete countable group and $(H,\tau)$ be a metric compactification of $G$.   The following statements are equivalent:
\begin{enumerate}
\item   $x\in\{0,1\}^G$ is almost automorphic on  $H$.   
\item $x\in\{0,1\}^G$ is the model set associated to a proper, irredundant, generic window $W\subseteq H$.
\end{enumerate}
  \end{corollary}

\subsection{Toeplitz arrays as model sets}\label{sec: Toeplitz as model set}
In Section~\ref{sec: semicocycles} we have proven for all odometers (which do not necessarily carry a group structure) a one-to-one correspondence between almost automorphic points in $\Sigma^G$ and semicocycles.
In Section~\ref{sec: almost auto as model set} we have shown for all metric compactifications (which need not be odometers) a one-to-one correspondence between almost automorphic points in $\{0,1\}^G$ and model sets.
Thus, for odometers with a group structure we immediately obtain a characterization of Toeplitz $0$-$1$ arrays, which we summarize in the following statements:

 \begin{theorem}\label{thm: Toeplitz is model set}
    Let $G$ be a residually finite, discrete, countable group. Let $\overleftarrow{G}$ be a free odometer with group structure, and let $\mathrm{Toep(\overleftarrow{G})}\subseteq \{0,1\}^G$ be the collection of all Toeplitz arrays $x\in \{0,1\}^G$ such that $(\overleftarrow{G},G)$ is the maximal equicontinuous factor of $(\overline{O_G(x)},G)$. Then, the map
    $\Lambda: \mathcal{W}(\overleftarrow{G})\to \mathrm{Toep(\overleftarrow{G})}$
    given by
    $$
    \Lambda(W)=x_W, \mbox{ for every } W\in \mathcal{W}(\overleftarrow{G}),
    $$
    is well defined and bijective.  Moreover, the boundary $\partial W$ of $W\in \mathcal{W}(\overleftarrow{G})$ satisfies
    $$ \nu(\partial W) = 1 - d,$$
    where $\nu$ is the normalized Haar measure on $\overleftarrow{G}$ and $d= \lim_{n\to\infty} d_n$ is the regularity constant of $x_W$, as defined in Section~\ref{sec: regularity}.
\end{theorem}
\begin{proof}
The elements in $\mathrm{Toep}(\overleftarrow{G})$ are exactly the arrays $x\in \{0,1\}^G$ which are almost automorphic on $\overleftarrow{G}$. Thus, this theorem is a direct application of  Theorem~\ref{thm: almost-automorphic-model-set} applied to $G$ and $H=\overleftarrow{G}$.\\
For the last part, note that by Proposition~\ref{Characterization_Phi} and Theorem~\ref{thm: chara. Toeplitz semicocycle} we have
$$\nu(\{\xi\in\overleftarrow{G}: y(1_G)=y^\prime(1_G) \text{ for all }y,y^\prime\in\beta^{-1}(\{\xi\})\}) = d,$$
so that $\nu(\partial W) = 1-d$ follows from \eqref{eq: chara boundary} in Theorem~\ref{thm: almost-automorphic-model-set}.
\end{proof}


The following result is a consequence of Theorems  \ref{thm: chara. Toeplitz semicocycle} and \ref{thm: Toeplitz is model set}.
\begin{corollary}
  Let $G$ be a residually finite, discrete, countable group and $(\overleftarrow{G},G)$ be a free odometer such that $\overleftarrow{G}$ has group structure. The following statements are equivalent:
\begin{enumerate}
\item $x\in\{0,1\}^G$ is a Toeplitz array such that $(\overleftarrow{G},G)$ is the maximal e\-qui\-con\-ti\-nuous factor of $(\overline{O_G(x)},G)$.

\item The map $x:G\to \{0,1\}$ is a semicocycle on $\overleftarrow{G}$   invariant under no rotation.

\item $\{g \in G \mid x(g)=1\}$ is a model set associated to a proper, irredundant, generic window $W\subseteq \overleftarrow{G}$.
\end{enumerate}
\end{corollary}

\section{Self similar windows}\label{sec: self similarity}
In this section, we will introduce and refine the concept of self similarity from \cite[Section 4.1]{FuGlJäOe21}, which is a sufficient condition for finiteness of the fibers. Our main result, Proposition~\ref{prop: k-self similar with total order implies k+1 elements in fiber} below, will allow us to precisely control the cardinality of the fibres. This criterion, in combination with Theorem~\ref{thm: Toeplitz is model set}, provides the basis for the construction of irregular Toeplitz subshifts with prescribed fibre cardinality in the next section.\\
In what follows, we assume as in the previous section that $G$ is a countable discrete group and that $(H,\tau)$ is a metric compactification of $G$, equipped with a bi-invariant metric $d$ (for instance, if $H$ is an odometer,  the metric introduced in Section~\ref{odometers} is bi-invariant). 
Recall that since (left and right) multiplications are homeomorphisms in every topological group, we have in particular that $\mathrm{int}(\xi_1 W \xi_2) = \xi_1\mathrm{int}(W)\xi_2$ and $\partial(\xi_1W\xi_2) = \xi_1 (\partial W) \xi_2$ for all $\xi_1,\xi_2\in H$ and $W\tm H$. We will introduce the concepts of similarity in a manner analogous to the definitions given for the abelian case in \cite{FuGlJäOe21}.  \medskip

Let $W\subseteq H$ be a proper, irredundant window, and $\beta:\overline{O_G(x_W)}\to H$ be the factor map given in Proposition~\ref{prop: torus para bis 1}. For every $\xi\in H$ and $x\in \overline{O_G(x_W)}$, Proposition~\ref{prop: torus para bis 1} ensures that 
 $$ \beta(x) = \xi   \iff \tau(G)\cap \mathrm{int}(W)\xi^{-1}\subseteq \{\tau(g): x(g)=1\}\subseteq \tau(G)\cap W\xi^{-1}.$$
    Hence, for $x,x^\prime\in\beta^{-1}(\{\xi\})$ and each $g\in G$ such that $x(g)=1$ and $x^\prime(g)=0$, we have  that $\tau(g) \in \partial (W\xi^{-1})$.
In particular, in the case $\partial (W\xi^{-1})\cap \tau(G)=\emptyset$, we see that $\#\beta^{-1}(\{\xi\})=1$.
The following concepts will allow to address the case in which $\tau(G)$ has non-empty intersection with $\partial(W\xi^{-1})$.  
\begin{definition}[Self similarity]\label{similarity}
    We call $\xi\in H$ {\bf critical} if $\partial (W\xi^{-1})\cap \tau(G)\neq \emptyset$.
    Let $\xi$ be critical and $\tau(g_1),\tau(g_2)\in \partial (W\xi^{-1})\cap \tau(G)$.
    We write $\tau(g_1) \preccurlyeq_{\xi} \tau(g_2)$ if there exists $\varepsilon>0$ such that
    \begin{equation}\label{eq: self-similarity-order} \tau(g_1)^{-1}\cdot (B_\varepsilon(\tau(g_1))\cap W\xi^{-1})\tm \tau(g_2)^{-1}\cdot (B_\varepsilon(\tau(g_2))\cap W\xi^{-1}).\end{equation}
    If both $\tau(g_1) \preccurlyeq_{\xi} \tau(g_2)$ and $\tau(g_2) \preccurlyeq_{\xi} \tau(g_1)$ hold, we write $\tau(g_1) \sim_{\xi} \tau(g_2)$ and call $\tau(g_1)$ and $\tau(g_2)$ {\bf similar with respect to $\xi$}. When $\tau(g_1)\preccurlyeq_{\xi} \tau(g_2)$, but $\tau(g_1) \not\sim_{\xi} \tau(g_2)$, we write $\tau(g_1) \prec_{\xi} \tau(g_2)$.\\
    For each fixed $\xi$, the relation $\sim_{\xi}$ defines an equivalence relation on $\partial (W\xi^{-1})\cap \tau(G)$.
    We call its equivalence classes {\bf similarity classes with respect to $\xi$}.
    On the similarity classes with respect to $\xi$ the relation $\preccurlyeq_{\xi}$ defines a partial ordering .
    If there are only finitely many similarity classes with respect to $\xi$ for all $\xi\in H$, we call the window $W$ {\bf self similar}.
    If the number of similarity classes is bounded by $k$ for all $\xi$, then we say that $W$ is {\bf $k$-self similar}.
    Lastly, if $W$ is $1$-self similar, then we also call it {\bf perfectly self similar}.
\end{definition}

\begin{lemma}[{\cite[Lemma 4.2]{FuGlJäOe21}}]\label{lem: if l_1 in Gamma then l_2}
    Let $\xi\in H$ be critical and assume that $\tau(g_1), \tau(g_2) \in \partial (W\xi^{-1})\cap \tau(G)$ satisfy $\tau(g_1)\preccurlyeq_{\xi} \tau(g_2)$. 
    Then for each $x\in\beta^{-1}(\{\xi\})$ we have that $x(g_1)=1$ implies $x(g_2)=1$.
\end{lemma}
\begin{proof}
Let $x\in \beta^{-1}(\{\xi\})$ and $(h_ix_W)_{i\in\mathbb{N}}$ be a sequence that converges to $x$. Note that $x\in\beta^{-1}(\{\xi\})$ implies $\tau(h_i)\to \xi$. Choose $\epsilon>0$ as in Definition~\ref{similarity}, so that (\ref{eq: self-similarity-order}) holds. Then $x(g_1)=1$ implies that for all sufficiently large $i$ we have $$\tau(g_1h_i)\xi^{-1}=\tau(g_1)\tau(h_i)\xi^{-1}\in B_\epsilon(\tau(g_1))\cap W\xi^{-1} . $$ Due to \eqref{eq: self-similarity-order}, this entails $$\tau(g_2)\tau(h_i)\xi^{-1}\in B_\epsilon(\tau(g_2))\cap W\xi^{-1} $$
and thus $\tau(g_2h_i)\in W$. Hence, in the limit $i\to\infty$ we obtain $x(g_2)=1$ as required. 
\end{proof}

The previous lemma immediately yields and upper estimate for the size of fibers:
\begin{corollary}[{\cite[Lemma 4.3]{FuGlJäOe21}}]\label{cor: number of fibers if self similar}
    If $W$ is self similar, then $\# \beta^{-1}(\{\xi\}) <\infty$ and if $W$ is $k$-self similar, then $\#\beta^{-1}(\{\xi\})\leq 2^k$ for all $\xi\in H$.
\end{corollary}
\begin{proof}
Suppose there are exactly $k$ similarity classes $S_1,\ldots, S_k$ with respect to $\xi$.
Then, by above lemma, for every $x\in \beta^{-1}(\{\xi\})$ there exists $(s_1,\ldots,s_k)\in\{0,1\}^k$ such that $\tau(g)\in S_i \implies x(g)=s_i$.
Since by Proposition~\ref{prop: torus para bis 1}, $x$ is determined on $\{g\in G: \tau(g) \notin \partial(W\xi^{-1})\}$, we see that $(s_1,\ldots,s_k)$ completely determines $x$.
Hence, $\#\beta^{-1}(\{\xi\})\leq \#\{0,1\}^k\leq 2^k$.
\end{proof}

We therefore see that using the cardinality of the set $\mathcal{S}$ of similarity classes one can obtain a straightforward upper bound on the size of the fiber.
On the other hand, it will turn out that using information about the partial order $\preccurlyeq_{\xi}$ on $\mathcal{S}$, we are able to both tighten the upper bound and also establish lower bounds.
In order to prove this, we need a criterion, which for a certain class of ``candidate elements'' $x\in \{0,1\}^G$ characterizes whether these lie in $\beta^{-1}(\{\xi\})$ or not.
 


The following characterization of the fibers, which is contained implicitely in \cite[Section 3.2]{JLO19} and explicitely in a more general setting in \cite{DJL}, enables us to do so:
 \begin{proposition}[\cite{DJL, JLO19}]\label{prop: criterion being in fiber}
    Let $W\subseteq \overleftarrow{G}$ be a proper, irredundant window. 
    Let $\xi\in \overleftarrow{G}$, and $x\in \{0,1\}^G$ be such that 
    \begin{align*}
        \tau(G)\cap \mathrm{int}(W)\xi^{-1}\subseteq \{\tau(g):x(g)=1\}\subseteq \tau(G)\cap W\xi^{-1}.
    \end{align*}
    Then, $x\in \overline{O_G(x_W)}$ if and only if for all finite sets $N,M\subseteq G$ such that $N\subseteq \{g\in G:x(g)=1\}$ and $M\subseteq \{g\in G:x(g)=0\}$ we have
    \begin{align}\label{eq: xi-closed}
        \xi\in \overline{\kl \bigcap_{l\in N} {\tau(l)}^{-1}W \setminus \bigcup_{j\in M}  \tau(j)^{-1}W  \kr \cap  \tau(G)}.
    \end{align}
\end{proposition}
The following proof is structurally the same as in \cite{DJL}. We include it for the convenience of the reader. 
\begin{proof}
    First, assume that \eqref{eq: xi-closed} is fulfilled. We show that that for every pair of finite sets $N\subseteq \{g\in G:x(g)=1\}$ and $M\tm \{g\in G:x(g)=0\}$ there exists $x_{N,M}\in\overline{O_G(x_W)}$ such that $N\subseteq \{g\in G:x_{N,M}(g)=1\}\subseteq G\setminus M$. This implies that $x\in \ol{O_G(x_W)}$ as well.\\
    Let $N,M$ be as above. Due to \eqref{eq: xi-closed}, there exists a sequence $(g_n)_{n\in\mathbb{N}}\subseteq G$ such that $\tau(g_n)\to \xi$ and 
    \begin{align}\label{eq: tau(g_n) in cap without cup}
        \tau(g_n)\in \bigcap_{l\in N} {\tau(l)}^{-1}W \setminus \bigcup_{j\in M}  \tau(j)^{-1}W. 
    \end{align}
    By going over to a subsequence if necessary, we can assume that the sequence $(g_n x_W)_{n\in\N}$ converges to some element $x_{N,M}$. Let $l\in N$ and $n\in\mathbb{N}$. 
    Then, $\tau(l)\tau(g_n)\in W$. Therefore, $1=x_W(lg_n)=(g_nx_W)(l)$, which implies in the limit $n\to \infty$ that $x_{N,M}(l)=1$ and we conclude that $l\in \{g\in G:x_{N,M}(g)=1\}$.
    
    Now, if $j\in M$, then (\ref{eq: tau(g_n) in cap without cup}) yields $\tau(j)\tau(g_n)\notin W$.
    Consequently, $x_W(jg_n)=0$.
    In the limit, we thus obtain $x_{N,M}(j)=0$. 
    
    Altogether, this yields $N\subseteq \{g\in G:x_{N,M}(g)=1\}\subseteq G\setminus M$, as required.\smallskip

    Now, assume that $x\in \overline{O_G(x_W)}$,  and $N,M\subseteq G$ are two finite sets such that $N\subseteq \{g\in G: x(g)=1\}$ and $M\subseteq \{g\in G: x(g)=0\}$.
    As $x\in\beta^{-1}(\{\xi\})$ by Proposition~\ref{prop: torus para bis 1}, there exists a sequence $(g_n)_{n\in\mathbb{N}}\subseteq G$ such that $g_nx_W\to x$ and $\tau(g_n)\to \xi$.
    Consequently, there is some $n_0\in\mathbb{N}$ such that $\tau(N)\subseteq W\tau(g_n)^{-1}\cap \tau(G)$ and $\tau(M)\cap W\tau(g_n)^{-1}=\emptyset$ for each $n\geq n_0$.
    Thus, for all $l\in N$, it holds that $\tau(g_n)\in \tau(l)^{-1}W$, and for each $j\in M$, it is true that $\tau(g_n)\notin \tau(j)^{-1}W$.
    Hence, for every $n\geq n_0$ we have 
    \begin{align*}
        \tau(g_n)\in \left(\bigcap_{l\in N}\tau(l)^{-1}W\setminus \bigcup_{j\in M}\tau(j)^{-1}W\right)\cap \tau(G).
    \end{align*}
    As $\tau(g_n)\to\xi$, this implies \eqref{eq: xi-closed}.
\end{proof}

In the following, for a window $W$ and finite subsets $N,M\tm G$ we define
$$ T_W(N,M) :=\bigcap_{l\in N} {\tau(l)}^{-1}W \setminus \bigcup_{j\in M}  \tau(j)^{-1}W $$
We will use the following sufficient criterion in later constructions:
\begin{lemma}\label{lem: suff crit being in fiber}
    Let $W$, $\xi$ and $x$ be as in Proposition~\ref{prop: criterion being in fiber}.
    If for all finite sets $N_0,M_0\tm G$ such that $N_0\tm\{g\in G: x(g) = 1\}$, $M_0\tm \{g\in G:x(g)=0\}$ as well as $\tau(N_0)\cup\tau(M_0)\tm \partial(W\xi^{-1})$ and all $\varepsilon>0$ the set $B_{\varepsilon}(\xi)\cap T_W(N_0,M_0)$ has non-empty interior, then $x\in \ol{O_G(x_W)}$ (and hence $x\in \beta^{-1}(\{\xi\})$).
\end{lemma}
\begin{proof}
    Let $N\tm \{g\in G : x(g)=1\}$ and $M\tm \{g\in G: x(g)=0\}$ be finite sets.
    We write $N = N_0\uplus N_1$ and $M = M_0\uplus M_1$, where $N_0 = \{g\in N: \tau(g) \in \partial(W\xi^{-1})\}$ and similarly for $M_0$.
    Observe that $T_W(N,M) = T_W(N_0,M_0)\cap T_W(N_1,M_1)$.
    Note that the assumption
    $$\tau(G)\cap \mathrm{int}(W)\xi^{-1}\subseteq \{\tau(g):x(g)=1\}\subseteq \tau(G)\cap W\xi^{-1}$$
    yields $\tau(N_1)\tm \mathrm{int}(W)\xi^{-1}$ and $\tau(M_1)\cap W\xi^{-1} = \emptyset$.
    Now, since $N_1,M_1$ are still finite and $W$ is closed, it is straightforward to see that $T_W(N_1,M_1)$ is a neighborhood of $\xi$.
    In particular, for all $\varepsilon>0$ small enough, we have $B_\varepsilon(\xi)\tm T_W(N_1,M_1)$, so that 
    $$ B_\varepsilon(\xi)\cap T_W(N,M) = B_{\varepsilon}(\xi)\cap T_W(N_0,M_0)$$
    has non-empty interior by assumption.
    As $\tau(G)$ is dense in $H$, we see that $\tau(G) \cap B_\varepsilon(\xi)\cap T_W(N,M) \neq \emptyset$.
    Since $\varepsilon$ can be arbitrarily small, $x\in \ol{O_G(x_W)}$ follows by Proposition~\ref{prop: criterion being in fiber}
\end{proof}

Observe that one can rewrite Definition~\ref{similarity} as
$$ \tau(g_1)\preccurlyeq_\xi \tau(g_2)\iff \exists\varepsilon>0: B_{\varepsilon}(\xi)\cap \tau(g_1)^{-1}W \tm B_{\varepsilon}(\xi) \cap \tau(g_2)^{-1}W.$$
Hence, we immediately obtain the following:

\begin{lemma}\label{lem: need only check at extrema}
    Let $N,M\tm G$ be finite sets such that $\tau(N)\cup\tau(M)\tm \partial(W\xi^{-1})$.
    Suppose, there exists $l_0\in N$ and $j_0\in M$ such that $\tau(l_0)\preccurlyeq_\xi \tau(l)$ for all $l\in N$ and $\tau(j_0)\succcurlyeq_\xi \tau(j)$ for all $j\in M$.\\
    Then, $B_{\varepsilon}(\xi)\cap T_W(N,M) = B_{\varepsilon}(\xi)\cap T_W(\{l_0\},\{j_0\})$ for all $\varepsilon>0$ small enough.
\end{lemma}
This allows us to take a first step towards a lower bound for the fiber.
\begin{lemma}\label{lem: partition yields fiber element}
    Let $W\tm H$ be a proper, irredundant window and let $\xi\in H$ be critical.
    Suppose, we can decompose the set $\mathcal{S}$ of similarity classes w.r.t.\ $\xi$ into $\mathcal{S}= \mathcal{S}^-\uplus \mathcal{S}^+$ such that $\mathcal{S}^-$ has a maximum $S_1$ (w.r.t.\ $\preccurlyeq_\xi$) and $\mathcal{S}^+$ has a minimum $S_2$.
    Assume further $S_1 \prec_\xi S_2$.
    If we define $x\in\{0,1\}^G$ via
    $$ x(g) = 1 \iff \tau(g) \in \mathrm{int}(W)\xi^{-1} \cup \bigcup_{S\in \mathcal{S}^+} S,$$
    then $x\in \beta^{-1}(\{\xi\})$.
    The result also holds true if either $\mathcal{S}^-$ or $\mathcal{S}^+$ is empty.
\end{lemma}
\begin{proof}
    Assume first that $\mathcal{S}_1$ and $\mathcal{S}_2$ are non-empty.
    In view of Lemma~\ref{lem: suff crit being in fiber}, let $N\tm\{g\in G: x(g)=1\}$ and $M\tm \{g\in G: x(g)=0\}$ be finite such that $\tau(N)\cup\tau(M)\tm \partial(W\xi^{-1})$ and let $\varepsilon>0$.
    By definition of $x$, we have $\tau(N)\tm \bigcup_{S\in \mathcal{S}^+}S$ and $\tau(M)\tm \bigcup_{S\in\mathcal{S}^-}S$.
    By making $N$ and $M$ larger if necessary, we can assume $\tau(l_0) \in S_2$ and $\tau(j_0)\in S_1$ for some $l_0\in N$, $j_0\in M$.\\
    Now, by $\tau(l_0)\prec_\xi \tau(j_0)$ and properness of the window, we obtain that  for any $\delta>0$ there exists some 
    $$\zeta\in (B_{\delta}(\xi)\cap \tau(l_0)^{-1}W) \setminus (B_{\delta}(\xi)\cap \tau(j_0)^{-1}W).$$
    Using $W = \ol{\mathrm{int}(W)}$, we can even find an open set
    $$ U_\delta \tm (B_{\delta}(\xi)\cap \tau(l_0)^{-1}W) \setminus (B_{\delta}(\xi)\cap \tau(j_0)^{-1}W) = B_{\delta}(\xi)\cap T_W(\{l_0\},\{j_0\}).$$
    Moreover, as $S_2$ is the minimum of $\mathcal{S}^+$, it follows that $\tau(l_0)\preccurlyeq_\xi \tau(l)$ for all $l\in N$.
    Similarly, we have $\tau(j_0)\succcurlyeq_\xi \tau(j)$ for all $j\in M$.
    Therefore, Lemma~\ref{lem: need only check at extrema} implies for all sufficiently small $\delta>0$ (w.l.o.g.\ $\delta\leq \varepsilon$) that
    $$ U_\delta \tm B_{\delta}(\xi)\cap T_W(\{l_0\},\{j_0\}) = B_\delta(\xi)\cap T_W(N,M) \tm B_\varepsilon(\xi) \cap T_W(N,M).$$
    Hence, $x\in \beta^{-1}(\{\xi\})$ follows from Lemma~\ref{lem: suff crit being in fiber}.
    The cases $\mathcal{S}^+ = \emptyset$ or $\mathcal{S}^- = \emptyset$ are treated in an analogous way. 
\end{proof}

As a consequence of Lemma~\ref{lem: if l_1 in Gamma then l_2} and Lemma~\ref{lem: partition yields fiber element} we can completely determine the structure of the fibre $\beta^{-1}(\{\xi\})$ in the case that $\mathcal{S}$ is totally ordered:

\begin{proposition}\label{prop: k-self similar with total order implies k+1 elements in fiber}
    Let $\xi\in H$ be a critical point.
    Assume that the set $\mathcal{S}$ of similarity classes with respect to $\xi$ is $\mathcal{S} = \set{S_1,\ldots,S_k}$, where $S_i \prec_{\xi} S_j$ if $i< j$.
    Then, we have $\beta^{-1}(\{\xi\}) = \set{x_1,\ldots,x_{k+1}}$, with 
    $$ x_j(g) = 1 \iff  \tau(g) \in \mathrm{int}(W)\xi^{-1} \cup \bigcup_{i=j}^k S_i,\; \mbox{ for } j\in \{1,\ldots, k+1\}.$$ 
    (Here, we use the usual convention that $\bigcup_{j=k+1}^k S_j=\emptyset$.) In particular, $\#\beta^{-1}(\{\xi\}) = k+1$.
\end{proposition}
 
\begin{proof}
    The inclusion $\{x_1,\ldots,x_{k+1}\}\tm \beta^{-1}(\{\xi\})$ follows from Lemma~\ref{lem: partition yields fiber element} so that equality is obtained with Lemma~\ref{lem: if l_1 in Gamma then l_2}.
\end{proof}

\section{Construction of $k$-self similar windows}\label{sect: self-similar windows}
The aim of this section is to construct, for a given odometer $\overleftarrow{G}$ and any $k\in\mathbb{N}$, proper, generic, irredundant and irregular $k$-self similar windows $W^{(k)}\subseteq\overleftarrow{G}$ such that 
\begin{itemize}
    \item the similarity classes for each element $\xi\in\overleftarrow{G}$ are strictly ordered;
    \item there are exactly $k$ similarity classes for almost every $\xi\in\overleftarrow{G}$.
\end{itemize} In the light of Proposition~\ref{prop: k-self similar with total order implies k+1 elements in fiber}, the resulting Toeplitz flows will then have maximal rank $k+1$ with exactly $k+1$ elements in almost every fibre. Hence, this construction will prove Theorem~\ref{thm: existence of k-1 Toeplitz arrays}. 

In order to do so, we will first construct a perfectly self-similar window $W_\mathrm{perf}\subseteq\overleftarrow{G}$ in Section~\ref{subsect: perfectly self similar windows}. The same window will then provide the basis for the more refined construction of $k$-self similar windows in Section~\ref{subsect: k-self similar windows}, where we obtain the $W^{(k)}$ as modifications of $W_\mathrm{perf}$. The remaining details for the proof of Theorem~\ref{thm: existence of k-1 Toeplitz arrays}, including in particular the description of the measure-theoretic structure of the Toeplitz flows, are then discussed in Section~\ref{subsect: proof of main theorem}.

In the whole section, we assume that $G$ is an infinite, countable, amenable and residually finite group, and $\overleftarrow{G}$ is a $G$-odometer defined by a decreasing sequence $(\Gamma_n)_{n\in\mathbb{N}}$ of normal finite index subgroups of $G$ with trivial intersection. In particular, $\overleftarrow{G}$ is a metric compactification of $G$, so that all the results of the previous sections apply. 

We call $\tau:G\to \overleftarrow{G}$ the injective homomorphism such that $\tau(G)$ is dense in $\overleftarrow{G}$, and denote by $\nu$ the normalized Haar measure on $\overleftarrow{G}$. Recall that $\xi=(\xi_n\Gamma_n)_{n\in\mathbb{N}}$ and $\xi'=(\xi_n'\Gamma_n)_{n\in\mathbb{N}}$  in $\overleftarrow{G}$ satisfy $d(\xi, \xi')\leq 2^{-k}$ if and only if $\xi_k\Gamma_k=\xi_k'\Gamma_k$. We will denote by $[\xi]_k$ the ball with radius $2^{-k}$ around $\xi\in \overleftarrow{G}$, where $k\in\mathbb{N}$ and $\xi\in \overleftarrow{G}$, and refer to $[\xi]_k$ as a cylinder of level $k$. Observe that there is a one-to-one correspondence between cylinders of level $k$ and the elements of $G/\Gamma_k$.

\subsection{Construction of a perfectly self similar window} \label{subsect: perfectly self similar windows}
As said above, in this section we will be concerned with the construction of a perfectly self similar window $W=W_\mathrm{perf}$. We first give the definition of $W$ depending on a sequence of partitions of the sets $D_n\cap \Gamma_{n-1}$. In the following, we then provide different criteria that allow to ensure that the resulting window has the desired properties of properness, genericity, irredundancy, irregularity and perfect self similarity, respectively. 
Finally, in the proof of Proposition~\ref{prop: W_perf}, we then show that all these criteria can be fulfilled simultaneously. \medskip

In this section, we use the notation of Section \ref{subsect: expansion}.
For every $j\in \mathbb{N}$, recall the maps $\pi_j:G\to D_j\cap \Gamma_{j-1}$ defined in \eqref{eq: expansion coefficients} by $\pi_j(g) := \varphi_{j-1}(\psi_j(g))$, where the maps $\varphi_j:G\to \Gamma_j$ and $\psi_j:G\to D_j$ are defined as the unique elements $\varphi(g)\in\Gamma_j$ and $\psi_j(g)\in D_j$ that decompose $g\in G$.
That is, $g=\psi_j(g)\varphi_j(g)$ for every $g\in G$. \medskip

Given $n\in\mathbb{N}$, let $A_n, B_n, C_n$ be non-empty sets that form a partition of $D_n\cap \Gamma_{n-1}$, that is,  
 $$A_n\uplus B_n \uplus C_n = D_n\cap \Gamma_{n-1}.$$  
In particular, since we set $\Gamma_0=G$, we have that $\{A_1,  B_1, C_1\}$ is a partition of $D_1$.
We now define a window $W=W_\mathrm{perf}$ in several steps. First, we let   
$$
X_1=\bigcup_{a\in A_1}[\tau(a)]_1, \hspace{5mm}  Y_1=\bigcup_{a\in B_1}[\tau(a)]_1, \hspace{5mm} Z_1=\bigcup_{a\in C_1}[\tau(a)]_1.
$$
Once we have defined $X_n, Y_n, Z_n$, we further let 
$$
A_{n+1}^*=\{g\in D_{n+1}: \tau(g)\in Z_n \mbox{ and } {\pi_{n+1}(g)}\in A_{n+1}\}, 
$$
$$
B_{n+1}^*=\{g\in D_{n+1}: \tau(g)\in Z_n \mbox{ and } {\pi_{n+1}(g)}\in B_{n+1}\}, 
$$ 
$$
C_{n+1}^*=\{g\in D_{n+1}: \tau(g)\in Z_n \mbox{ and } {\pi_{n+1}(g)}\in C_{n+1}\}, 
$$
and 
$$
X_{n+1}=\bigcup_{a\in A_{n+1}^*}[\tau(a)]_{n+1},$$
$$Y_{n+1}=\bigcup_{a\in B_{n+1}^*}[\tau(a)]_{n+1}$$
$$Z_{n+1}=\bigcup_{a\in C_{n+1}^*}[\tau(a)]_{n+1}.
$$
Then we set
$$
U=\bigcup_{n\in\mathbb{N}}X_n  \hspace{5mm} \mbox{ and } \hspace{5mm} V=\bigcup_{n\in\mathbb{N}}Y_n.
$$
and finally define our window $W=W_\mathrm{perf}$ as 
\begin{equation}
    \label{eq: definition W}
    W = \overline{U}.
\end{equation}
We first establish that $W$ is indeed a proper window. 
\begin{lemma}\label{lem: alternative chara W and partial W}
The set $W\subseteq \overleftarrow{G}$ defined in (\ref{eq: definition W}) is proper, with $\mathrm{int}(W)=U$ and 
 \begin{equation}\label{eq: boundary characterisation}
 \partial W=\bigcap_{n\in\mathbb{N}}Z_n=\overleftarrow{G}\setminus (V\cup U).
 \end{equation}
\end{lemma}
\begin{proof}
Both $U$ and $V$ are open as unions of cylinder sets. As these sets are disjoint, this implies $V\subseteq \overline{W}^c$ and $\partial W\subseteq \overleftarrow{G}\setminus (U\cup V)$. In order to prove the converse inclusion, suppose that $\xi=(\xi_n\Gamma_n)_{n\in\mathbb{N}}\in \overleftarrow{G}\setminus (U\cup V)$, where $\xi_n\in D_n$ for each $n\in\mathbb{N}$, and fix $n\in\N$. 
Note that since $\overleftarrow{G}\setminus(U\cup V)=\bigcap_{n\in\N} Z_n$, we have $\xi\in Z_n$. 
As $Z_n$ is a union of cylinder sets of level $n$, this implies $[\tau(\xi_n)]_n\subseteq Z_n$. Hence, if we choose $a\in A_{n+1}$ and $b\in B_{n+1}$ and let 
$g=\xi_na$ and  $h=\xi_nb$, we have $\tau(g)\in X_{n+1}\subseteq U$ and $\tau(h)\in Y_{n+1} \subseteq V$. Since both $\tau(g)$ and $\tau(h)$ have distance less than $2^{-n}$ from $\xi$ and $n\in\N$ was arbitrary, we obtain $\xi\in\partial W$. This shows  \eqref{eq: boundary characterisation}, which further implies $U=\mathrm{int}(W)$ and thus $W=\overline{\mathrm{int}(W)}$ as required.   
 \end{proof}   

In the following lemmas, we will establish criteria that allow to choose the partitions $\{A_n,B_n,C_n\}$ in the above construction in such a way that the resulting window $W$ is generic, irredundant, irregular and perfectly self-similar.
We start with genericity. 

\begin{lemma}\label{genericity}
The window $W$ is generic if and only if $1_G\notin C_n$ for infinitely many $n\in\N$. 
\end{lemma}
\begin{proof}
By Lemma~\ref{lem: alternative chara W and partial W}, we have  
$$ \partial W = \bigcap_{n\in\N} Z_n  = \left\{\xi\in \overleftarrow{G}: \pi_n(\xi) \in C_n \text{ for all } n\in \N\right\}.$$
Moreover, using property (3) of Lemma~\ref{lem: ex. of seq. of fundamental domains}, we see that $\xi\in \tau(G)$ if and only if $\pi_n(\xi) \neq 1_G$ for only finitely many $n$.
This yields `$\Leftarrow$'.

Conversely, suppose that $1_G\in C_n$ for all but finitely many $n\in\N$. Choose a sequence $(g_n)_{n\in\N}$ with $g_n\in C_n$ and $g_n=1_G$ for all but finitely many $n$. Then $g=\prod_{n\in\N} g_n$ is well-defined and satisfies $\tau(g)\in\partial W$, so that $W$ is not generic. This shows `$\Rightarrow$'. 
\end{proof}

Regarding irredundancy, we will use the following sufficient criterion, which can also be applied to general windows in $\overleftarrow{G}$. Given a cylinder set $S\subseteq \overleftarrow{G}$ of level $n\in\N_0$, let 
$$ \mathcal{M}_n(W,S) =\{S'\tm S:S' \text{ cylinder set of level } n+1 \text{ and } S' \tm \overleftarrow{G}\setminus W\}.$$
Thereby, we understand $\overleftarrow{G}$ to be the unique level $0$ cylinder set in $\overleftarrow{G}$. 
\begin{lemma}\label{lem: suff. crit. irredundancy}
    Assume that for every $n\in\N_0$ there exists a cylinder set $S\tm \overleftarrow{G}$ of level $n$  which satisfies $S\cap W \neq \emptyset$, $S\cap (\overleftarrow{G}\setminus W)\neq \emptyset$ and $\# \mathcal{M}_n(W,S) = 1$.
    Then $W$ is irredundant.
\end{lemma}
\begin{proof}
    Let $\xi=(\xi_n\Gamma_n)_{n\in\mathbb{N}}\in \overleftarrow{G}$  such that $W\xi=W.$ We need to show that $\xi=1_{\overleftarrow{G}}$, which means that $\xi_n\in \Gamma_n$ for all $n\in\N$. In order to do so, we proceed by induction on $n$.
    
    Let $S_1$ be the unique cylinder set of level $1$ such that $S_1\tm \overleftarrow{G}\setminus W$.
    Then, $S_1\xi \tm \overleftarrow{G}\setminus W$, as well, so that $ S_1\xi = S_1$.
    This is equivalent to $\xi_1 \in \Gamma_1$.
    
    Assume now that we have shown $\xi_j \in \Gamma_j$ for all $j\in \set{1,\ldots, n}$.
    Take some cylinder set $S_n$ of level $n$ such that $S_n\cap W \neq \emptyset$, $S_n\cap (\overleftarrow{G}\setminus W)\neq \emptyset$ and $\# \mathcal{M}_n(W,S_n) =1$.
    Since we are assuming $\xi_n\in  \Gamma_n$, we have $ S_n\xi = S_n$.
    Hence, for the unique element $S_{n+1}$ of $\mathcal{M}_n(W,S_n)$ we have both $ S_{n+1}\xi\tm S_n$ and $S_{n+1}\xi \tm \overleftarrow{G}\setminus W$.
    Therefore, $S_{n+1}\xi = S_{n+1}$ and thus $\xi_{n+1} \in\Gamma_{n+1}$, finishing the proof.
\end{proof}
For our specific window $W$ defined by \eqref{eq: definition W}, we have an easy characterization of the conditions of the previous lemma.
The proof is left to the reader.
\begin{lemma}\label{lem: chara nonempty intersect with W and complement}
    Let $C\tm \overleftarrow{G}$ be a cylinder of level $n$.
    The following statements are equivalent:
    \begin{enumerate}
        \item $C\cap W\neq \emptyset \neq  C\cap (\overleftarrow{G}\setminus W)$.
        \item $C\cap Z_n \neq \emptyset$.
        \item $C\tm Z_n$.
        \item $C\cap \partial W\neq \emptyset$.
    \end{enumerate}
    Furthermore, the following statements are equivalent:
    \begin{enumerate}
        \item $\#\mathcal{M}_n(W,C)=1$ for every cylinder $C$ of level $n$ with $C\cap W\neq \emptyset \neq  C\cap (\overleftarrow{G}\setminus W)$.
        \item $\# B_{n+1} =1$.
    \end{enumerate}
\end{lemma}

This leads to the following simple sufficient criterion for irredundancy.
\begin{corollary}\label{cor: suff crit irredundancy}
    If $\#B_n= 1$ for all $n\in\N$, then $W$ is irredundant.
\end{corollary}

Since the sequence of sets $(Z_n)_{n\in\mathbb{N}}$ is decreasing, we can easily compute the measure of the boundary $\partial W$. This will allow to ensure the irregularity of the window. 

\begin{lemma}\label{lem: measure of boundary}
    Let $W\subseteq \overleftarrow{G}$ be the set defined in (\ref{eq: definition W}). Then
    $$\nu(\partial W) = \lim_{n\to\infty} \frac{1}{\# D_n} \prod_{k=1}^n \# C_k = \lim_{n\to\infty} \prod_{k=1}^n \frac{\# C_k}{\#( D_k \cap \Gamma_{k-1})}.$$
\end{lemma}
\begin{proof}
Since the sequence of sets $(Z_n)_{n\in\mathbb{N}}$ is decreasing and we have $\partial W=\bigcap_{n\in\N} Z_n$ by  Lemma~\ref{lem: alternative chara W and partial W}, we obtain $\nu(\partial W)=\lim_{n\to \infty}\nu(Z_n)$. The measure of any cylinder set of level $n$ is equal to $\frac{1}{\#D_n}$, and by construction $Z_n$ contains $\prod_{k=1}^n \# C_k$ of such cylinder sets. Together, this yields 
\[
\nu(Z_n)=\frac{1}{\#D_n}\prod_{k=1}^n\#C_k=\prod_{k=1}^n\frac{\#C_k}{\#(D_k\cap \Gamma_{k-1})}.  \qedhere
\]
\end{proof}

If $\#B_n = 1$ and $1_G \notin C_n$ for every $n \in \mathbb{N}$, then Lemmas~\ref{lem: alternative chara W and partial W} and~\ref{genericity} together with Corollary~\ref{cor: suff crit irredundancy} imply that the window $W$ resulting from the above construction is proper, generic and irredundant. In combination with Theorem~\ref{thm: Toeplitz is model set}, this yields that the model set $x_W \in \{0,1\}^G$ is a Toeplitz array. Moreover, if $\nu(\partial W)>0$, then $x_W$ is irregular. In the light of Lemma~\ref{lem: measure of boundary}, this will be the case when $\#(A_n\cup B_n)$ grows sufficiently slow compared to $\frac{\# D_{n+1}}{\#D_n}=\#(D_{n+1}\cap \Gamma_n)$. It remains to provide an efficient criterion for the perfect self-similarity of the window. 
\smallskip

In order to do so, we will exploit the fact that $G$ is amenable, which so far has not been needed in this section. 
Further, we use the fact mentioned above that for every $g\in G$ there exists $n\in\N$ such that $\pi_j(g)=1_G$ for all $j>n$. Hence, the expansion $g = \prod_{j=1}^\infty \pi_j(g)$ can always be understood as a product of finitely many group elements. We note that this is true only due to the particular choice of the fundamental domains $D_n$ according to Lemma~\ref{lem: ex. of seq. of fundamental domains}, namely the fact that $G=\bigcup_{n\in\N} D_n$. This is unlike the standard $p$-adic expansion, where for instance the expansion of $-1\in \Z$ is given by $(p-1,p^2-1,\ldots)$.

Applying an inductive argument, we have that the sets $C_{n+1}^*$ can be written as 
\begin{align*}
    C_{n+1}^*=\{g\in D_{n+1} : \pi_j(g)\in C_j, 1\leq j \leq n+1\}.
\end{align*}
Therefore, we have that $\xi\in Z_n$ if and only if $\pi_j(\xi)\in C_j$ for every $j\in \{1,\ldots, n\}$.
\begin{lemma}\label{perfectly-self-similar}
    Let $W\subseteq \overleftarrow{G}$ be the set defined in (\ref{eq: definition W}) and suppose that $W$ is irredundant. Moreover, assume that for every $n\in\mathbb{N}$ we have
\begin{equation} \label{eq: C_n is in the inside of D_n}
 K_n\cdot C_n \tm D_n,
\end{equation}
where $K_n$ is the range of the map $d_n:G\times G \to G$ defined in Remark~\ref{rem: carry-over maps}. Then $W$ is perfectly self similar.
\end{lemma}
\begin{proof}
Let $\xi \in \overleftarrow{G}$ and $t^{(1)},t^{(2)} \in \partial(W\xi^{-1})\cap \tau(G)$.
Note that in particular we have $\eta:= t^{(2)}(t^{(1)})^{-1}\in \tau(G)$.
Let $M\in\N$ be such $\pi_n(\eta) = 1_G$ for all $n\geq M$.
We claim that 
$$ (t^{(1)})^{-1}\cdot\kl[t^{(1)}]_M \cap (V\xi^{-1})\kr  = (t^{(2)})^{-1}\cdot\kl[t^{(2)}]_M \cap (V\xi^{-1})\kr,$$
which implies, due to $W=\overleftarrow{G}\setminus V$ that $t^{(1)}$ and $t^{(2)}$ are similar with respect to $\xi$.
To that end, let $\zeta \in [t^{(1)}]_M\cap V\xi^{-1}$.
By $\zeta \in V\xi^{-1}$ and definition of $V$, there exists $N\in\N$ such that $\pi_n(\zeta\xi) \in C_n$ for all $n<N$ and $\pi_N(\zeta\xi) \in B_N$.
However, due to $\zeta \in [t^{(1)}]_M$ and $t^{(1)} \in \partial (W \xi^{-1})$, we obtain for all $n\leq M$ that
$$ \pi_n(\zeta\xi) = \pi_n( t^{(1)}\xi) \in C_n.$$
Hence, $N>M$.
Now, in order to prove that $ t^{(2)}(t^{(1)})^{-1}\zeta =\eta\zeta\in V \xi^{-1} $, we need to compute $\pi_n(\eta\zeta\xi)$.
First, note that since $\zeta \in [t^{(1)}]_M$ it follows that $\eta\zeta\xi= t^{(2)}(t^{(1)})^{-1}\zeta\xi \in [t^{(2)}\xi]_M$.
Thus, as before we obtain 
$$ \pi_n(\eta\zeta\xi) = \pi_n(t^{(2)}\xi) \in C_n$$
for all $n\leq M$.
For the case $n>M$, we make use of Lemma~\ref{lem: carry over rule} and Remark~\ref{rem: extension of psi_j and pi_j} with $g=\psi_n(\eta)$ and $h=\psi_n(\zeta\xi)$, from which we obtain 
\begin{align*}
    c_M =   d_M\cdot \alpha_{\psi_{M-1}(\zeta\xi)}(\underbrace{\pi_M(\eta)}_{=1_G})\cdot \pi_M(\zeta\xi) 
    = d_M\cdot \pi_M(\zeta\xi) 
\end{align*}
with $d_M \in K_M$.
Now it follows by $\pi_M(\zeta\xi)\in C_M$ and condition (\ref{eq: C_n is in the inside of D_n}) that $d_M\cdot \pi_M(\zeta\xi)  \in D_M$, so that we obtain $d_{M+1}= \varphi_M(c_M)= 1_G$ for the following carry-over.
Therefore, we readily compute $\pi_{M+1}(\eta\zeta\xi) = \pi_{M+1}(\zeta\xi)$ and by induction $\pi_n(\eta\zeta\xi) = \pi_n(\zeta\xi)$ for all $n>M$.
In particular, we have shown that $\pi_n(\eta\zeta\xi) \in C_n$ for all $n<N$ as well as $\pi_N(\eta\zeta\xi) \in B_N$.
This shows $\eta \zeta  \in V\xi^{-1} $ as desired and hence 
$$ (t^{(1)})^{-1}\cdot\kl [t^{(1)}]_M \cap (V\xi^{-1}) \kr   \tm (t^{(2)})^{-1}\cdot\kl [t^{(2)}]_M \cap (V\xi^{-1}) \kr.$$
The other inclusion follows by symmetry, finishing the proof.
\end{proof}

It remains to show that the above conditions for genericity, irredundancy, self simililarity and irregularity can all be fulfilled simultaneously. This is done in the proof of the following proposition, which summarises the results of this section.  

\begin{proposition}\label{prop: W_perf}
Suppose that $G$ is an infinite countable amenable residually finite group and $\overleftarrow{G}$ is a $G$-odometer defined by a decreasing sequence $(\Gamma_n)_{n\in\mathbb{N}}$ of normal finite index subgroups with trivial intersection. 
Let $\varepsilon\in(0,1)$. Then there exists a window $W_{\mathrm{perf}}$ which is is proper, irredundant, generic, perfectly self similar and satisfies $\nu(\partial W_{\mathrm{perf}}) \geq 1-\varepsilon$.
\end{proposition} 

\begin{remark}
    This immediately implies that there exists an irregular Toeplitz flow $\left(\overline{O_G(x)},G\right)$ with maximal equicontinuous factor $\overleftarrow{G}$ such that almost all fibres over the odometer have cardinality two, and there are no fibres of larger cardinality. 

    In order to see this, we simply let $x=x_W$ and apply Corollary~\ref{cor: number of fibers if self similar} with $k=1$ to see that the maximal cardinality of a fibre is $2$. Proposition \ref{prop: regular Toeplitz flows} then yields that almost every fibre has cardinality $2$. Note here that the set $\mathcal{T}$ of injectivity points in $\overline{O_G(x)}$ is $G$-invariant, and since the Toeplitz flow is irregular it must have measure~$0$. 
\end{remark}

\begin{proof}[Proof of Proposition~\ref{prop: W_perf}] Recall that given finite sets $A,K\subseteq G$, the (left) $K$-boundary (Van Hove boundary) $\partial_K(A)$ of $A$ with respect to $K$ is defined as $$ \partial_K(A) = \{g\in G: K^{-1}\cdot g \text{ intersects both } A \text{ and } G\setminus A\}.$$
Since $(D_n)_{n\in\N}$ is a left F\o lner sequence and $G$ is discrete, we have 
\begin{equation} \label{eq: Foelner property} \lim_{n\to\infty} \frac{\# \partial_K(D_n)}{\# D_n} = 0 
\end{equation}
for all finite sets $K$. Further, observe that if $C_n\subseteq D_n$ and $1_G\in K_n$, condition~\eqref{eq: C_n is in the inside of D_n} is equivalent to $C_n \tm G\setminus \partial_{K_n^{-1}}(D_n)$.
\smallskip

Now, fix some sequence $(a_n)_{n\in\N}$ of integers $a_n\geq 2$. As discussed in Remark~\ref{rem: carry-over maps}, 
the set $K_n=d_n(G\times G)$ only depends on the sets $D_1,\ldots,D_{n-1}$. Therefore, using \eqref{eq: Foelner property} and going over to a subsequence if necessary, we can assume that the quantities $\# \partial_{K_n^{-1}}(D_n)/{\# D_n}$ and hence also 
$$
\frac{\#\partial_{K_n^{-1}}(D_n)+a_n}{\#(D_n\cap \Gamma_{n-1})} \ = \ \frac{(\# \partial_{K_n^{-1}}(D_n)+a_n) \cdot \# D_{n-1}}{\# D_n}
$$ are arbitrarily small. In particular, given any $\delta>0$, we can assume that 
$$\frac{\#\partial_{K_n^{-1}}(D_n)+a_n}{\#(D_n\cap \Gamma_{n-1})} \ < \ 1- \exp\kl{-\frac{\delta}{2^n}}\kr.$$
Therefore, we can choose the partition $\{A_n,B_n,C_n\}$ of $D_n\cap \Gamma_{n-1}$ such that 
\begin{eqnarray}
   & C_n  \tm  (D_n\cap \Gamma_{n-1})\setminus \partial_{K_n^{-1}}(D_n), \label{eq: boundary condition satisfied}\\
&\# \kl\kl (D_n\cap \Gamma_{n-1})\setminus \partial_{K_n^{-1}}(D_n)\kr \setminus C_n \kr= a_n, \label{eq: boundary condition estimate}
\end{eqnarray}
and at the same time $\#B_n=1$ and $\#A_n \geq a_n-1\geq 1$ (where we use $a_n\geq 2$ in the case that $\partial_{K_n^{-1}}(D_n)=\emptyset$) and $1_G\notin C_n$. 

Now, as mentioned above, \eqref{eq: boundary condition satisfied} implies that $C_n$ satisfies \eqref{eq: C_n is in the inside of D_n}. At the same time, as $\# B_n=1$, the window $W$ is irredundant by Corollary~\ref{cor: suff crit irredundancy}. Therefore, we can apply Lemma~\ref{perfectly-self-similar} to see that $W$ is perfectly self-similar. Moreover, properness and genericity follow from Lemmas \ref{lem: alternative chara W and partial W} and \ref{genericity}, respectively.  It remains to prove the lower bound on $\nu(\partial W)$. To that end, note that
$$
\frac{\# C_n}{\# (D_n\cap \Gamma_{n-1})} 
\ \geq \ \frac{\# (D_n\cap \Gamma_{n-1}) -(\# \partial_{K_n^{-1}}(D_n) +a_n)}{\# D_n\cap \Gamma_{n-1}} \ \geq \ \exp\kl{-\frac{\delta}{2^n}}\kr
$$
by \eqref{eq: boundary condition estimate}. Choosing $\delta = -\log(1-\varepsilon)$, Lemma~\ref{lem: measure of boundary} now yields $\nu(\partial W) \geq \exp({-\delta})=1-
\varepsilon$ as required.
\end{proof}

\begin{remark}
    \label{rem: number of retained cylinders}
    For further use in the next subsection, the following observation is important. By $\mathcal{C}_n$, we denote the family of cylinder sets of level $n$. In the above proof, we can choose the sequence $(a_n)_{n\in\N}$ such that $a_n\geq 3$ for all $n\in\mathbb{N}$. This will entail that $\# A_n\geq 2$ for all $n\in\N$. By construction, it then follows that for all $C\in \mathcal{C}_n$ with $C\cap \partial W_{\mathrm{perf}}\neq \emptyset$ there exist at least two different cylinder sets $C^{(1)}$ and $C^{(2)}$ in $\mathcal{C}_{n+1}$ such that $C^{(1)}\uplus C^{(2)}\tm C \cap X_{n+1}$. 

    This fact will play a crucial role in the construction of the $k$-self similar windows $W^{(k)}$ below as modifications of $W_\mathrm{perf}$. 
\end{remark}

\begin{remark} 
We remark that despite the fact that the employed methods are quite different, the Toeplitz arrays obtained from the construction in this section can be seen as a generalisation of  well-known examples of Toeplitz sequences presented by Williams for the case $G=\Z$ in \cite[Section 4]{Wi84}. 
\end{remark}

\subsection{Construction of $k$-self similar windows with strictly ordered similarity classes} \label{subsect: k-self similar windows}
We now aim to modify the window $W_\mathrm{perf}$ obtained in the last section in order to produce $k$-self similar windows $W^{(k)}$ with strictly ordered similarity classes. To that end, for the whole section, we fix $\epsilon\in (0,1)$ and let $W_{\mathrm{perf}}\tm \overleftarrow{G}$ be the window given by Proposition~\ref{prop: W_perf}. Recall that we have $\nu(\partial W_\mathrm{perf})>0$ and $\# B_n = 1$ for all $n\in\N$. An important point of the construction is that we will keep the boundary of $W_\mathrm{perf}$, that is, we will obtain $\partial W^{(k)}=\partial W_\mathrm{perf}$. \medskip

Since $\nu(\partial W_\mathrm{perf})>0$, there exist $L\in\mathbb{N}$ and a clopen partition $\{H_1,\ldots, H_k\}$ of $\overleftarrow{G}$ such that  $H_j$ is a union of cylinder sets of level $L$ and $\nu(\partial W_\mathrm{perf} \cap H_j) >0$, for all $j=1,\ldots k$. In addition, we also choose some partition $\N = N_1\uplus \ldots \uplus N_k$ of the integers such that $\# N_j = \infty$ for all $j$ and $\set{1,\ldots, L}\subseteq N_1$. 
For convenience, we let $N_{[i,j]} = N_i\cup\ldots\cup N_j$, where $1\leq i\leq j\leq k$.

The following statement will allow us to modify the window $W_\mathrm{perf}$, by removing certain collections of cylinder sets from its interior, while still keeping the same boundary. It will also be instrumental in Section~\ref{sect: infinite max rank}. Recall that $\mathcal{C}_n$ denotes the collection of all cylinder sets of level $n$. 

\begin{lemma}\label{lem: modified window boundary} 
    Suppose that $(E_n)_{n\in\N}$ is a sequence of cylinder sets such that $E_n \in \mathcal{C}_n$ and $E_n\tm X_n$, where $X_n$ is defined as in Section~\ref{subsect: perfectly self similar windows}.
    Let $\widetilde{X}_n^{(0)} = X_n$ and $\widetilde{X}_n^{(1)}=X_n \setminus E_n$.
    Then, for any sequence $(s_n)_{n\in\N} \in \{0,1\}^\N$ and any infinite set $\mathcal{N}\tm \N$, we have 
    $$ \partial \kl \bigcup_{n\in \mathcal{N}} (H_j \cap {X}_n^{(s_n)})\kr = \partial W_{\mathrm{perf}} \cap H_j \quad \text{for all } 1\leq j\leq k.$$
    In particular, the window
    $$ \widetilde{W} := \partial W_{\mathrm{perf}} \cup \bigcup_{j=1}^k \bigcup_{n\in N_{[1,j]}} (H_j\cap {X}_n^{(s_n)})$$
    is proper with $\partial(\widetilde{W}\cap H_j) = \partial W_{\mathrm{perf}}\cap H_j$ for all $j=1,\ldots, k$ and $\partial\widetilde{W} = \partial W_{\mathrm{perf}}$.
\end{lemma}
\begin{proof}
    In order to prove the first claim, we let $W_{\mathcal{N},j}:= \bigcup_{n\in \mathcal{N}}(H_j \cap \widetilde{X}_n^{(s_n)})$. Further, for every $n\in\N$ we choose $e_n\in\overleftarrow{G}$ such that $E_n = [e_n]_n$.
        Let $\xi \in \partial W_{\mathrm{perf}}\cap H_j$, so that $\xi\in Z_n$ for every $n\in\N$ by Lemma~\ref{lem: alternative chara W and partial W}.
    Since $H_j$ is open and $Z_n$ is a union of cylinders of level $n$, we can find $m\in\N$ such that $[\xi]_n\tm H_j\cap Z_n$ for all $n\geq m$. 
    
    Fix $n\geq m$. As $\mathcal{N}\tm \N$ is infinite, there exists $\widetilde{n}\geq n$ such that $\widetilde{n}+1 \in \mathcal{N}$.
    We choose some $a_{\widetilde{n}+1}\in A_{\widetilde{n}+1}\setminus\{\pi_{\widetilde{n}+1}(e_{\widetilde{n}+1})\}$, which is possible due to $\#A_{\widetilde{n}+1}\geq 2$.
    If we define $g = \kl\prod_{j=1}^{\widetilde{n}}\pi_j(\xi)\kr\cdot a_{\widetilde{n}+1}$, then $\tau(g) \in [\xi]_n\cap(\widetilde{X}_{\widetilde{n}+1}^{(s_{\widetilde{n}+1})}\cap H_j)$ by construction.
    As $n\geq m$ was arbitrary, this shows $\xi\in \ol{W_{\mathcal{N},j}}$.
    Together with $\xi\in \ol{\overleftarrow{G}\setminus W_{\mathrm{perf}}} \tm \ol{\overleftarrow{G}\setminus W_{\mathcal{N},j}}$, we obtain $\xi \in \partial W_{\mathcal{N},j}$.\\
    Conversely, if $\xi \notin \partial W_{\mathrm{perf}}$, then we can find $n\in\N$ minimal such that $\xi\notin Z_n$.
    In this case, it is easy to see that either $\xi\in X_n$ or $\xi \in Y_n$.
    If $\xi \in Y_n$ or $\xi \in X_n\setminus \widetilde{X}_n^{(s_n)}$, then clearly $[\xi]_n \tm \overleftarrow{G}\setminus W_{\mathcal{N},j}$, so that $\xi \notin \partial W_{\mathcal{N},j}$.
    If $\xi \in \widetilde{X}_n^{(s_n)}$, then obviously $\xi \in W_{\mathcal{N},j} = \mathrm{int}(W_{\mathcal{N},j})$, so that $\xi\notin \partial W_{\mathcal{N},j}$ in this case, as well.
    This shows the first claim.\\
    The fact that $\partial(\widetilde{W}\cap H_j) = \partial W_{\mathrm{perf}} \cap H_j$ is now obtained easily from the above with $\mathcal{N} = N_{[1,j]}$.
    Furthermore, $\partial\widetilde{W}=\partial W_{\mathrm{perf}}$ follows now because the $H_j$ form a clopen partition of $\overleftarrow{G}$.
    This also shows properness of $\widetilde{W}$ and thus completes the proof.
\end{proof}




We define the modified windows $W^{(k)}$ as 
$$ W^{(k)}=\partial W_{\mathrm{perf}}\cup \bigcup_{j=1}^k\bigcup_{n\in N_{[1,j]}}(H_j\cap X_n) = W_{\mathrm{perf}} \setminus \bigcup_{j=1}^k \bigcup_{n\in N_{[j+1,k]}} (H_j\cap X_n).  $$
Note that this corresponds to the above construction of $\widetilde{W}$ with $s_n=0$ for all $n\in\N$ (the case of non-constant $(s_n)_{n\in\N}$ only being needed in Section~\ref{sect: infinite max rank}). Further, observe that cylinder sets of level $n$ with $n\in N_1$ are ``never removed'' from the interior of $W_\mathrm{perf}$ by this construction, and that $H_k \cap W^{(k)} = H_k \cap W_{\mathrm{perf}}$ (since $N_{[1,k]}=\N)$.
Lemma~\ref{lem: modified window boundary}  now immediately yields 
\begin{corollary} \label{cor: W^(k) boundary}
The window $W^{(k)}$ is proper with $\partial W^{(k)} = \partial W_{\mathrm{perf}}$.
\end{corollary}
We also note that as a consequence $W^{(k)}$ is both generic and irregular, since both holds for $W_\mathrm{perf}$ and these properties only depend on the boundary of the window. Therefore, it now remains to check that $W^{(k)}$ also has the two remaining desired properties of irredundancy and $k$-self similarity (with strictly ordered similarity classes). We start with irredundancy. 

\begin{lemma}\label{lemma: irredundant}
    The window $W^{(k)}$ is irredundant.
\end{lemma}
\begin{proof}
    In the light of Lemma~\ref{lem: suff. crit. irredundancy}, we need to show that for each $n\in\N$ there exists $C\in\cC_{n}$ such that $C\cap W^{(k)}\neq \emptyset \neq C\cap (\overleftarrow{G}\setminus W^{(k)})$ and $\#\mathcal{M}_n(W^{(k)},C)=1$. Due to Lemma~\ref{lem: chara nonempty intersect with W and complement} and the construction of $W_\mathrm{perf}$ (that is, the fact that $\# B_n=1$), we know that $\#\mathcal{M}_n(W_\mathrm{perf},C)=1$ holds for all $C\in\mathcal{C}_n$ with $C\cap \partial W_\mathrm{perf}\neq \emptyset$. We now distinguish two cases. 

    First, suppose that $n<L-1$ and choose some cylinder set $C\in \cC_n$ that intersects $\partial W_\mathrm{perf}$. Since $\partial W_\mathrm{perf}=\partial W^{(k)}$, this means that $C$ intersects both $W^{(k)}$ and its complement. At the same time, $C$ contains a unique cylinder set $C'$ of level $n+1$ that does not intersect $W_\mathrm{perf}$. Suppose $S$ is any other cylinder set of level $n+1$ contained in  $C$. Then either $S\subseteq W_\mathrm{perf}$, or $S\cap \partial W_\mathrm{perf}\neq \emptyset$. In the first case, we also have $S\subseteq W^{(k)}$, since no cylinder sets of level $n+1\leq L$ were removed from $W_\mathrm{perf}$ in the construction. In the second case, $S$ intersects $\partial W^{(k)}$. Hence, in both cases, $S$ cannot be contained in the complement of $W^{(k)}$ and we obtain $\# \cM_n(W^{(k)},C)=1$. 

    Secondly, suppose that $n\geq L$. In this case, we can choose a cylinder set $C\in\cC_n$ such that $C\cap\partial W_\mathrm{perf}\neq \emptyset$ and at the same time $C\subseteq H_k$.
    Then we have $\# \cM_n(W_\mathrm{perf}, C)=1$, but since $W^{(k)}\cap H_k = W_\mathrm{perf}\cap H_k$, this also means $\#\cM(W^{(k)},C)=1$. 

    Altogether, this shows that the assumptions of Lemma~\ref{lem: suff. crit. irredundancy} are satisfied and we obtain the irredundancy of $W^{(k)}$. 
\end{proof}

Now we turn our attention to $k$-self similarity and the ordering of similarity classes. In this context, a first observation that will be useful is the fact that the perfect self similarity of $W_\mathrm{perf}$ also holds for the sets $X_n$ that appear as intermediates in the construction of that window. The precise statement reads as follows. 

\begin{lemma}\label{lemma: nested-union}
Let $\xi_1, \xi_2\in \partial W_{\mathrm{perf}} = \partial W^{(k)}$ and $m\in\mathbb{N}$ be such that 
$$
\xi_1^{-1}\cdot\left([\xi_1]_m\cap W_{\mathrm{perf}} \right)=\xi_2^{-1}\cdot\left([\xi_2]_m\cap W_{\mathrm{perf}} \right)
$$
Then $$\xi_1^{-1}\cdot\left([\xi_1]_m\cap X_n \right)=\xi_2^{-1}\cdot\left([\xi_2]_m\cap X_n \right)$$ holds for every $n\in\mathbb{N}$.
\end{lemma}

\begin{proof}
First, observe that from the construction of $W_{\mathrm{perf}}$, it follows that for $C\in \mathcal{C}_n$,  
\begin{equation}\label{eq: C in W iff C in X_n}
    C\tm W_{\mathrm{perf}} \iff C\tm X_t, \mbox{ for some } 1\leq t\leq n.
\end{equation}
If $n\leq m$, then for $i=1,2$ we either have $[\xi_i]_m \cap X_n = \emptyset$ or $[\xi_i]_m \tm X_n$.
However, the latter case would imply $\xi_i \in X_n \tm \mathrm{int}(W_{\mathrm{perf}})$, contradicting the assumption $\xi_i\in \partial W_{\mathrm{perf}}$. Thus, we have $\xi_1^{-1}\cdot\left([\xi_1]_m\cap W_{\mathrm{perf}} \right)=\emptyset =\xi_2^{-1}\cdot\left([\xi_2]_m\cap W_{\mathrm{perf}} \right)
$, so that the assertions holds in this case (with both sets being empty).

Now, assume that $n> m$, so that $[\xi_i]_m\cap X_n$ is a (disjoint) union of cylinders of level $n$.
Let $C$ be one of these cylinders. 
Our hypothesis implies that $\xi_2\xi_1^{-1} C\subseteq [\xi_2]_m\cap W_{\mathrm{perf}}$, and the last set is  equal to $[\xi_2]_m\cap \bigcup_{j>m} X_j$ by the previous argument.
Applying \eqref{eq: C in W iff C in X_n}, there exists $m< t\leq n$ such that $\xi_2\xi_1^{-1} C\subseteq [\xi_2]_m\cap X_t$.
Let $C'\in \mathcal{C}_t$ be such that $\xi_2\xi_1^{-1} C\subseteq C' \subseteq [\xi_2]_m\cap X_t$.
An analogous argument applying \eqref{eq: C in W iff C in X_n} implies that there exists $m< t'\leq t$ such that $C\subseteq \xi_1\xi_2^{-1} C'\subseteq [\xi_1]_m\cap X_{t'}$.
Consequently, $t'=t=n$, concluding the proof.
\end{proof}

 
The following lemma now describes the structure of the similarity classes within $\partial W^{(k)}$.

\begin{lemma}\label{self-similar-k}
    Let $\xi \in \overleftarrow{G}$, $1\leq i_1 \leq i_2 \leq k$ and suppose that $\tau(l_1)  \in (\partial W^{(k)}\cap H_{i_1})\xi^{-1}\cap \tau(G)$ and $\tau(l_2) \in (\partial W^{(k)} \cap H_{i_2})\xi^{-1} \cap \tau(G)$.
    Then, $\tau(l_1) \sim_\xi \tau(l_2)$ if and only if $i_1=i_2$ and $\tau(l_1) \prec_\xi \tau(l_2)$ if and only if $i_1<i_2$. 
\end{lemma}
\begin{proof}
Let $l_1^*=\tau(l_1)$ and $l_2^*=\tau(l_2)$. Since $\partial W^{(k)}=\partial W_{\mathrm{perf}}$, perfect self similarity of $W_{\mathrm{perf}}$ (see Lemma~\ref{perfectly-self-similar}) implies there exists $m\in \mathbb{N}$ such that
$$
l_1^{*-1}\cdot([l^*_1]_m\cap W_{\mathrm{perf}}\xi^{-1})=l_2^{*-1}\cdot ([l^*_2]_m\cap W_{\mathrm{perf}}\xi^{-1}).
$$
Obviously, this implies $l^*_2l_1^{*-1}\cdot ([l^*_1]_m \cap \partial W_{\mathrm{perf}}\xi^{-1}) = [l^*_2]_m \cap \partial W_{\mathrm{perf}} \xi^{-1}$.
Moreover, from Lemma~\ref{lemma: nested-union}, taking $\xi_1=l^*_1\xi$ and $\xi_2=l^*_2\xi$, we deduce that 
 $$l^*_2l_1^{*-1}\cdot \left([l^*_1]_m\cap X_n\xi^{-1}\right)=[l^*_2]_m\cap X_n\xi^{-1},$$ for every $n\in \mathbb{N}$.
Note that since $l^*_s\in H_{i_s}\xi^{-1}$, we can assume $[l^*_{s}]_m\subseteq H_{i_s}\xi^{-1}$ (by choosing $m\geq L$) for $s=1,2$. Thus,
\begin{eqnarray*}
[l^*_2]_m\cap W^{(k)}\xi^{-1} & = & [l^*_2]_m\cap \left( \partial W_{\mathrm{perf}}\xi^{-1} \cup \bigcup_{j=1}^k\bigcup_{n\in N_{[1,j]}} H_j\xi^{-1}\cap X_n\xi^{-1}  \right)\\
   &=& \kl [l^*_2]_m\cap \partial W_{\mathrm{perf}} \xi^{-1}\kr\cup \bigcup_{n\in N_{[1,i_2]}}[l^*_2]_m\cap X_n\xi^{-1}\\
   &=& l^*_2l_1^{*-1}\kl ([l^*_1]_m \cap \partial W_{\mathrm{perf}}\xi^{-1}) \cup \bigcup_{n\in N_{[1,i_2]}} [l^*_1]_m \cap X_n\xi^{-1}\kr\\
   &\mt&l^*_2l_1^{*-1}\kl ([l^*_1]_m \cap \partial W_{\mathrm{perf}}\xi^{-1}) \cup \bigcup_{n\in N_{[1,i_1]}} [l^*_1]_m \cap X_n\xi^{-1}\kr\\
   &=& l^*_2l_1^{*-1}\cdot\left([l^*_1]_m\cap W^{(k)}\xi^{-1} \right)
\end{eqnarray*}
 with equality iff $i_1=i_2$ and strict inclusion iff $i_1<i_2$. This proves the assertion.  
\end{proof}

We summarize the above findings in the following
\begin{corollary} \label{cor: k similarity class structure}
    Given $\xi\in\overleftarrow{G}$ and $j\in\{1,\ldots ,k\}$, let 
    $$
    S_j^\xi = (\partial W^{(k)} \cap H_j)\xi^{-1} \cap \tau(G) . 
    $$
    Then each of the sets $S_j^\xi$ equals a (possibly empty) similarity class with respect to the relation $\sim_\xi$. Further, if $i<j$ and $S_i^\xi,S_j^\xi$ are non-empty, then $S_i^\xi\prec_\xi S_j^\xi$. 
\end{corollary}

\subsection{Proof of Theorem~\ref{thm: existence of k-1 Toeplitz arrays}.} \label{subsect: proof of main theorem}
We fix $k\geq 1$ and show in the following that the assertion of Theorem~\ref{thm: existence of k-1 Toeplitz arrays} for $k+1$ holds. \medskip

First, Co\-ro\-llary~\ref{cor: W^(k) boundary} implies that $W^{(k)}$ is proper, generic and irregular, whereas Lemma~\ref{lemma: irredundant} yields irredundancy. Consequently, Theorem~\ref{thm: Toeplitz is model set} implies that $x_{W^{(k)}}\in\{0,1\}^G$ is an irregular Toeplitz array. Further, by Proposition~\ref{prop: torus para bis 1} there exists a unique factor map $\beta: \overline{O_G(x_{W^{(k)}})}\to \overleftarrow{G}$ such that $\{x_{W^{(k)}}\}=\beta^{-1}(\{1_{\overleftarrow{G}}\})$. Due to Corollary~\ref{cor: k similarity class structure}, we can apply Lemma~\ref{lem: if l_1 in Gamma then l_2} and Proposition~\ref{prop: k-self similar with total order implies k+1 elements in fiber} to obtain the following. 
\begin{corollary}\label{cor: fiber structure for W^(k)}
 We have $\#\beta^{-1}(\{\xi\})\leq k+1$, for every $\xi\in \overleftarrow{G}$.
 Furthermore, $\# \beta^{-1}(\{\xi\}) = k+1$ if and only if $S_j^\xi := (\partial W^{(k)} \cap H_j)\xi^{-1} \cap \tau(G) \neq \emptyset$ for all $j\in \set{1,\ldots,k}$.
\end{corollary}
In particular, this settles property (1) of Theorem~\ref{thm: existence of k-1 Toeplitz arrays}.\medskip

In order to prove (2), let 
$$H_0 := \{\xi\in\overleftarrow{G}: \#\beta^{-1}(\{\xi\}) = k+1\}.
$$
Then Corollary~\ref{cor: fiber structure for W^(k)} yields $\xi \in H_0 \iff \forall j\in \{1,\ldots,k\}:\: S_j^\xi \neq \emptyset$.
Consider the action of $G$ on $\overleftarrow{G}$ given by $g\cdot \xi= \tau(g)\xi$, for every $g\in G$ and $\xi\in \overleftarrow{G}$.
This $G$-action is uniquely ergodic, where the unique ergodic measure  is the Haar measure $\nu$ on $\overleftarrow{G}$.
Hence, for any tempered (left-) F\o lner sequence $(F_n)_{n\in\N}$ and any $j=1,\ldots , k$, the pointwise ergodic theorem \cite{Lin01} implies  the existence of a full measure set $M_j$ of $\overleftarrow{G}$ such that for all $\xi \in M_j$ one has
\begin{equation} \label{eq: H_0 definition} \frac{1}{\#F_n} \sum_{l\in F_n} 1_{\partial W^{(k)} \cap H_j} (\tau(l)\xi ) \xrightarrow{n\to\infty} \nu(\partial W^{(k)} \cap H_j) > 0.
\end{equation}
In particular, for each $\xi \in M_j$ there exist infinitely many $l \in G$ such that $\tau(l) \xi  \in \partial W^{(k)} \cap H_j$, so that $\tau(l) \in S_j^\xi$. 
Of course, $\bigcap_{j=1}^k M_j$ is still a full measure set. 
As $\bigcap_{j=1}^k M_j \tm H_0$, we obtain $\nu(H_0)=1$ as desired. 

For later use, we also choose a full measure set $M_0\subseteq\overleftarrow{G}$ such that 
\begin{equation} \label{eq: M_0 definition} \frac{1}{\#F_n} \sum_{l\in F_n} 1_{\mathrm{int}(W^{(k)})} (\tau(l)\xi ) \xrightarrow{n\to\infty} \nu(\mathrm{int}(W^{(k)})    > 0
\end{equation}
and let 
\begin{equation} \label{eq: def hat H_0}
\widehat{H}_0 = \bigcap_{j=0}^k M_j.
\end{equation}\medskip

It thus remains to verify assertion (3) about the structure of the invariant measures of $(\overline{\cO(x_{W^{(k)}}}),G)$ in Theorem~\ref{thm: existence of k-1 Toeplitz arrays}. 
Proposition~\ref{prop: k-self similar with total order implies k+1 elements in fiber} implies that  for every $\xi\in H_0$ we have
$$
\beta^{-1}(\{\xi\}) = \{x^\xi_1 ,\ldots x^\xi_{k+1}\},
$$
where $x^\xi_j\in\{0,1\}^G$ is defined by 
$$
x_j^\xi(g)=1 \quad \Longleftrightarrow \quad \tau(g) \in \mathrm{int}(W^{(k)})\xi^{-1} \cup \bigcup_{i=j}^k S_i^\xi $$
for $j=1,\ldots , k+1$. If we let 
\begin{align*}
 W_j^{(k)} 
  \ = \ \mathrm{int}(W^{(k)}) \cup\bigcup_{i=j}^k(\partial W^{(k)} \cap H_i),
\end{align*}
then we see that 
$$
x_j^\xi = x_{W_j^{(k)}\xi^{-1}}$$
for $j=1,\ldots,k+1$. Hence, for every $\xi\in H_0$, the respective fibre is given by 
$$
  \beta^{-1}(\{\xi\})\ = \ \left\{ x_{W^{(k)}_1\xi^{-1}}, \ldots, x_{W^{(k)}_{k+1}\xi^{-1}}\right\}.
$$
Our aim now is to show that for each $j=1,\ldots k+1$ the mapping $$\phi_j : H_0 \to \overline{\cO(x_{W^{(k)}})} , \quad \xi\mapsto x_{W^{(k)}_j\xi^{-1}}$$ defines a measurable function on $H_0$, whose image in $\overline{\cO(x_{W^{(k)}})}$ is $G$-invariant and supports an ergodic $G$-invariant measure. The latter is obtained by projection the Haar measure $\nu$ on $\overleftarrow{G}$ to $\overline{\cO(x_{W^{(k)}})}$ via $\phi_j$. \medskip

We first ensure the measurability of the mappings $\phi_j$. 
\begin{lemma}\label{Borel_map}
For every $1\leq j\leq k+1$, the map $\phi_j:H_0\to \overline{O_G(x_{W^{(k)}})}=X$ defined above 
is Borel.
\end{lemma}
\begin{proof}
Let $n\in\N$ and $a_{n,1},\ldots, a_{n,k_n}\in G$ be such that $[a_{n,1}]_n,\ldots, [a_{n,k_n}]_{n}$  are  the different cylinder sets of level $n$ (where $k_n=[G:\Gamma_n]$). For every $1\leq j\leq k+1$, define 
$$W_{n,j}^{(k)}=\kl\mathrm{int}(W^{(k)})\setminus\bigcup_{i=1}^{j-1} (H_i \cap Z_n)\kr\cup \bigcup_{i=j}^k\left(H_i\cap Z_n\right),$$
and for every $\xi\in H_0$ and $g\in G$ we set $$\phi_{n,j}(\xi) = x_{W_{n,j}^{(k)}\psi_n(\xi)^{-1}},$$ which means that 
$$\phi_{n,j}(\xi)(g) = 1 \iff \tau(g) \in W_{n,j}^{(k)}\psi_n(\xi)^{-1}.$$

Since  $\phi_{n,j}(H_0)=\left\{x_{W_{n,j}^{(k)}a_{n,s}^{-1}}: 1\leq s\leq k_n\right\}$ and
$$(\phi_{n,j})^{-1}\left(\left\{x_{W_{n,j}^{(k)}a_{n,s}^{-1}}\right\}\right)=[a_{n,s}]_n,$$ for every $1\leq s\leq k_n$ and $n\in \N$, the maps $\phi_{n,j}$ are Borel.
It is therefore enough to show that $\lim_{n\to \infty}\phi_{n,j}(\xi)=\phi_j(\xi)$, for every $\xi\in H_0$. 

Let $\xi\in H_0$ and $g\in G$.
There are four cases to distinguish:
\begin{enumerate}
    \item $\phi_j(\xi)(g)=1$ and $\tau(g)\xi \in \partial W^{(k)}$.
    \item $\phi_j(\xi)(g)= 1 $ and $\tau(g) \xi\notin \partial W^{(k)}$.
    \item $\phi_j(\xi)(g)=0$ and $\tau(g)\xi \in \partial W^{(k)}$. 
    \item $\phi_j(\xi)(g)=0$ and $\tau(g)\xi\notin \partial W^{(k)}$.
\end{enumerate}
In case (1), we see that $\tau(g)\xi \in \bigcup_{i=j}^k(\partial W^{(k)} \cap H_i)$.
Let $j\leq i \leq k$ such that $\tau(g)\xi \in \partial W^{(k)}\cap H_i$.
Due to Lemma~\ref{lem: alternative chara W and partial W}, this implies for every $n\geq L$ that $\tau(g)\xi\in Z_n \cap H_i$.
Now, we have $[\tau(g)\xi]_n = [\tau(g)\psi_n(\xi)]_n$.
Since $Z_n\cap H_i$ is a union of cylinder sets of level $n$, this also implies $\tau(g) \psi_n(\xi) \in Z_n\cap H_i \tm W_{n,j}^{(k)}$.
Therefore, we have $\phi_{n,j}(\xi)(g) = 1$ for all $n\geq L$, which shows that $\lim_{n\to \infty}\phi_{n,j}(\xi)(g)=\phi_j(\xi)(g)$ in this case.\\
In case (2) we obtain $\tau(g)\xi \in \mathrm{int}(W^{(k)})$.
Choose $m\in \N$ such that $[\tau(g)\xi]_m\tm \mathrm{int}(W^{(k)})$.
In particular, we have $[\tau(g)\xi]_m\cap \partial W_{\mathrm{perf}} = [\tau(g)\xi]_m \cap \partial W^{(k)} = \emptyset$.
Thus, Lemma~\ref{lem: chara nonempty intersect with W and complement} yields $[\tau(g)\xi]_m\cap Z_m=\emptyset$.
Altogether, we obtain for all $n\geq m$ that 
$$ \tau(g)\psi_n(\xi) \in [\tau(g)\xi]_m \tm \mathrm{int}(W^{(k)}) \setminus Z_m \tm W_{n,j}^{(k)},$$
so that $\phi_{n,j}(\xi)(g) = 1$. This again shows $\lim_{n\to \infty}\phi_{n,j}(\xi)(g)=\phi_j(\xi)(g)$.\\
In case (3) we deduce $\tau(g)\xi\in \bigcup_{i=1}^{j-1}(\partial W^{(k)}\cap H_i)$ and can proceed in an analougous way as in (1) to show $\phi_{n,j}(\xi)(g) = 0$ for all $n\geq L$. Finally, in case (4) we obtain $\tau(g)\xi \in \overleftarrow{G}\setminus W^{(k)}$ which is analogous to (2).

In conclusion, we have shown that $\phi_j:H_0\to X$ is a pointwise limit of Borel maps, so that $\phi_j$ itself is Borel.
\end{proof}

Now we are in the position to complete the proof of Theorem~\ref{thm: existence of k-1 Toeplitz arrays} with the following
 \begin{proposition}\label{prop: measure theoretic chara of W^(k)}
    The dynamical system $\overline{O_G(x_{W^{(k)}})}=X$ has exactly $k+1$ ergodic measures $\mu_1,\ldots, \mu_{k+1}$, where $\mu_j$ is supported on the set
   $$
   \Omega_j = \phi_j\left(\widehat{H}_0\right) . 
   $$
   Moreover, $(X, \mu_j,G)$ is measure-theoretically conjugate to $(\overleftarrow{G}, \nu,G)$, for every $1\leq j\leq k+1$. 
 \end{proposition}
    \begin{proof}
     Let $1\leq j\leq k+1$. We first note that $\Omega_j$ is a measurable set. This can be seen as follows: 
     first, we have that $\Omega=\bigcup_{j=1}^{k+1} \Omega_j=\beta^{-1}\left(\widehat{H}_0\right)$ is measurable as the preimage of a measurable set under the continuous mapping $\beta$. Moreover, due to \eqref{eq: H_0 definition}, \eqref{eq: M_0 definition}, and the definition of the sequences $x_{W^{(k)}_j\xi^{-1}}$, the limits 
     $$
     \eta(x) = \lim_{n\to\infty} \frac{1}{\#F_n} \sum_{g\in F_n} x(g) 
     $$
     exist for all $x\in\Omega$. Note here that any $x\in\Omega$ is of the form $x=x_{W^{(k)}_j\xi^{-1}}$ for some $j=1,\ldots k+1$ and $\xi\in\overleftarrow{G}$ and we have 
     $$
     \eta\left(x_{W^{(k)}_j\xi^{-1}}\right) = \nu\left(\mathrm{int}(W^{k})\cup\bigcup_{i=j}^k (\partial W^{(k)}\cap H_j)\right) =: d_j
     $$
     due to \eqref{eq: H_0 definition} and \eqref{eq: M_0 definition}. For this reason, we have 
     $$
     \Omega_j= \Omega \cap \eta^{-1}(\{d_j\}) ,
     $$ 
     which is clearly measurable since $\eta$ is a measurable function.
     Lemma~\ref{Borel_map} now implies that $\beta_{|\Omega_j}: \Omega_j \to \widehat{H}_0$ is a bi-measurable bijective map with inverse $\phi_j$. Moreover, since 
     $$
     g\phi_j(\xi)=gx_{W^{(k)}_j\xi^{-1}}=x_{W^{(k)}_j\xi^-1g^-1}=\phi_j(g\xi),
     $$
     this map is equivariant and $\Omega_j$ is $G$-invariant. Consequently, the measure $\mu_j(A)=\nu(\beta(A\cap \Omega_j))$ is a well defined invariant probability measure on the Borel subsets of $X$. Moreover, this is an ergodic measure because $\nu$ is ergodic.    

     The measures $\mu_1,\ldots, \mu_{k+1}$ are different, because their supports are disjoint. 
     From Corollary~\ref{cor: fiber structure for W^(k)} and \cite[Lemma 3.6]{HLSY21} or \cite[Prop. 3.1 and Add. 3.2]{BreHaJa25},  we finally deduce that the ergodic measures of $X$ are exactly $\mu_1,\ldots, \mu_{k+1}$. 
    \end{proof}

\begin{remark} We note that the structure of the system $(X,G)$ in Proposition \ref{prop: measure theoretic chara of W^(k)} implies that its topological sequence entropy is exactly equal to $\log(k+1)$. This can be seen as follows: as a consequence of \cite[Proposition 3.7]{HLSY21}, $(X,G)$ has non-trivial $k+1$-IN-tuples, but only trivial $k+2$-IN-tuples (as a consequence of the fact that an $n$-IN-tuple is an $n$-regionally proximal tuple, see for instance \cite[Proposition 2.6]{GoLeMu25}).
This in turn implies, joint with the last part of Proposition \ref{prop: measure theoretic chara of W^(k)} and \cite[Theorem 4.4]{HuYe09}, that $X$ is a system having zero topological entropy but having topological sequence entropy equal to $\log(k+1)$. The same remark applies when IN-tuples are replaced by IT-tuples.
\end{remark}



\section{Measure-theoretically $k:1$ Toeplitz flows with infinite fibers} \label{sect: infinite max rank}

Our final goal now is to modify the windows $W^{(k)}$ constructed in the previous section in order to obtain windows $\widetilde{W}^{(k)}$ such that $\partial \widetilde{W}^{(k)} = \partial W^{(k)} = \partial W_{\mathrm{perf}}$ and for all $\xi \in\overleftarrow{G}$ the fibres of the factor map $\widetilde\beta:\overline{O_G(x_{\widetilde{W}^{(k)}})}\to\overleftarrow{G}$ are given by 
\begin{equation} \label{eq: new fibers}
\widetilde\beta^{-1}(\{\xi\}) = \left\{\widetilde{x}_1,\ldots,\widetilde{x}_{k+1}\right\} \cup \left\{\widetilde{x}_{k,l} \mid \tau(l) \in S_k^\xi\right\} ,
\end{equation}
where 
\begin{equation}\label{eq:tilde-x}
   \widetilde{x}_j(g)=1 \iff \tau(g) \in \mathrm{int}(\widetilde{W}^{(k)})\xi^{-1} \cup \bigcup_{i=j}^k S_i^\xi, 
\end{equation}
and the sequences $\widetilde{x}_{k,l}$ are defined by 
$$
\{g\in G: \widetilde{x}_{k,l}(g)=1\}=\{g\in G: \widetilde{x}_k(g)=1\}\setminus\{l\}.
$$
Here, we let $S_j^\xi = (\partial W^{(k)}\cap H_j)\xi^{-1}\cap \tau(G) = (\partial \widetilde{W}^{(k)}\cap H_j)\xi^{-1}\cap \tau(G)$ (see Corollary~\ref{cor: k similarity class structure}.\\
In particular, the dynamical system $(\ol{O_G(x_{\widetilde{W}^{(k)}})},G)$ has countably infinite maximal rank.
Note that if some of the sets $S_j$ are empty, then some of the fibre elements in \eqref{eq: new fibers} may coincide. In particular, if $\xi$ is non-critical, then $\beta^{-1}(\{\xi\}) = \{x_{\widetilde{W}^{(k)}\xi^{-1}}\}$.
However, we will show that for all $\xi$ from a full-measure subset $\widetilde{H}_0$, all the arrays included in \eqref{eq: new fibers} are distinct, so that the fiber $\widetilde{\beta}^{-1}(\{\xi\})$ is infinite.\\
Since the additional new elements $\widetilde{x}_{k,l}$ in the fibers of $\widetilde\beta$ differ only in one coordinate from $\widetilde{x}_k$, they are all asymptotic to each other. In particular, this means that our modification does not concern the measure-theoretic structure of the resulting systems. 
Hence, this will prove Theorem~\ref{thm: existence of k-1 Toeplitz arrays with infinite max rank}. 

We adopt all the notation and definitions from the previous section.
In order to construct $\widetilde{W}^{(k)}$, we choose one cylinder $E_n\subseteq X_n$ of level $n$ for each $n\in \N$ in such a way that 
\begin{equation}\label{eq: boundary of E_n}
    \partial \kl \bigcup_{n\in N_k} (X_n\setminus E_n)\kr = \partial \kl \bigcup_{n\in N_k} E_n\kr = \partial W_{\mathrm{perf}}.
\end{equation}
This can be achieved due to Remark~\ref{rem: number of retained cylinders}. 
In view of Lemma~\ref{lem: modified window boundary} we let $\widetilde{X}_n^{(0)} = X_n$ and $\widetilde{X}_n^{(1)} = X_n \setminus E_n$ and define the sequence $(s_n)_{n\in\N} \in \{0,1\}^{\N}$ via $s_n = 1 \iff n\in N_k$.
Now, we define
$$ \widetilde{W}^{(k)}= \partial W_{\mathrm{perf}} \cup \bigcup_{j=1}^k \bigcup_{n\in N_{[1,j]}} (H_j\cap \widetilde{X}_n^{(s_n)}).$$
Note that, as can be seen from the construction, the choice of the $E_n$ for $n\notin N_k$ is actually irrelevant and is only made for formal reasons.\smallskip

Lemma~\ref{lem: modified window boundary} now yields that $\widetilde{W}^{(k)}$ is proper, generic and irregular with $\partial \widetilde{W}^{(k)} = \partial W_{\mathrm{perf}} = \partial W^{(k)}$. In order to ensure the irredundancy of $\widetilde{W}^{(k)}$, we assume that the set $H_k$ in the construction of $W^{(k)}$ in the last section is chosen such that it contains at least two cylinder sets of level $L$ contained in $Z_L$ (this is always possible by choosing $L$ large enough). In this case, given any $n\geq L$, there exist at least two cylinder sets $C_i$, $i=1,2$, in $\cC_n$ which intersect the boundary of $W^{(k)}$ and satisfy $\cM_n(W^{(k)},C_i)=1$. Since at most one cylinder set of level $n+1$ is removed from $W^{(k)}$ in the construction of $\widetilde{W}^{(k)}$, we must have either $\cM_n(\widetilde{W}^{(k)},C_1)=1$ or $\cM_n(\widetilde{W}^{(k)},C_2)=1$. This allows to prove the irredundancy of $\widetilde{W}^{(k)}$ in the same way as in the proof of Lemma~\ref{lemma: irredundant}. 

Further, since we only remove cylinder sets of levels $n\in N_k$,  we immediately obtain that 
\begin{equation}\label{eq: new window structure 3}
\widetilde{W}^{(k)} \cap H_j  =  W^{(k)}\cap H_j
\end{equation}
for all $1\leq j < k$. 
Let $\xi\in \overleftarrow{G}$.
The following lemma aims to describe the relation $\preccurlyeq_\xi$ on $\partial(\widetilde{W}^{(k)}\xi^{-1})\cap \tau(G)$
\begin{lemma}\label{lem: order structure tilde(W)}
    Let $l_1,l_2,l_3\in G$ be distinct elements such that $\tau(l_i)\in \partial (\widetilde{W}^{(k)}\xi^{-1})$.
    Let $i_1,i_2,i_3\in \{1,\ldots,k\}$ such that $\tau(l_j)\in S_{i_j}^\xi$, with $S_i^\xi = (\partial \widetilde{W}^{(k)}\cap H_i)\xi^{-1}\cap \tau(G)$ as above.
    \begin{enumerate}
        \item If $i_1\leq i_2<k$, then $\tau(l_1)\preccurlyeq_\xi \tau(l_2)$.
        \item If $i_1<i_2\leq k$, then $\tau(l_1)\prec_\xi \tau(l_2)$.
        \item If $i_1=i_2 =k$, then we have neither $\tau(l_1)\preccurlyeq_\xi \tau(l_2)$ nor $\tau(l_2)\preccurlyeq_\xi \tau(l_1)$ -- both elements are not comparable via the order $\preccurlyeq_\xi$.
        \item If $i_1=i_2=i_3=k$, then there exists $\varepsilon>0$ such that
        $$
        \tau(l_1)^{-1}\cdot\left(B_\varepsilon(\tau(l_1))\cap \widetilde{W}^{(k)} \xi^{-1}\right) 
        \subseteq \bigcup_{i\in\{2,3\}}\tau(l_i)^{-1} \cdot\kl B_{\varepsilon}(\tau(l_i))\cap \widetilde{W}^{(k)}\xi^{-1}\kr.$$
    \end{enumerate}
\end{lemma}
\begin{proof}
    (1) and (2) follow from \eqref{eq: new window structure 3} and Lemma~\ref{self-similar-k}.
    (3) is a consequence of \eqref{eq: boundary of E_n} and Lemma~\ref{lemma: nested-union}.
    (4) follows from the fact that we remove at most one cylinder set of every level, so that we cannot remove corresponding cylinder sets sufficiently close to the $\tau(l_j)$, $j=1,2,3$.
\end{proof}

\begin{proposition}\label{prop: proof of eq: new fibers}
    The modified window $\widetilde{W}^{(k)}$ satisfies (\ref{eq: new fibers}). 
    Moreover, there exists a set $\widetilde{H}_0\tm \overleftarrow{G}$ of full Haar measure such that the fiber $\widetilde{\beta}^{-1}(\{\xi\})$ is infinite for all $\xi\in \widetilde{H}_0$.
\end{proposition}
\begin{proof}
To show that \eqref{eq: new fibers} is satisfied, let $\xi\in \overleftarrow{G}$ and let $S_j^\xi = (\partial \widetilde{W}^{(k)} \cap H_j)\xi^{-1}\cap \tau(G)$ for $1\leq j\leq k$, so that $\{ S_1^\xi,\ldots,S_k^\xi\}$ is a partition of $\partial (\widetilde{W}^{(k)}\xi^{-1})\cap \tau(G)$ (where some sets might be empty).
Further, for $l\in G$ such that $\tau(l)\in S_k^\xi$ we denote $S_{k,l}^\xi = \{\tau(l)\}$.\\
Suppose that $x$ is an element of $\widetilde{\beta}^{-1}(\{\xi\})$.
If we have $x(g)=0$ for all $g\in G$ such that $\tau(g) \in \partial (\widetilde{W}^{(k)}\xi^{-1})\cap \tau(G)$, then clearly $x = \widetilde{x}_{k+1}$.
Hence, we can assume that there exists a smallest index $j\in \{1,\ldots,k\}$ for which we can find $g\in G$ such that $\tau(g)\in S_j^\xi$ and $x(g)=1$.
If $j<k$, then it follows from Lemma~\ref{lem: order structure tilde(W)} and Lemma~\ref{lem: if l_1 in Gamma then l_2} that $x=\widetilde{x}_j$.\\
Thus, we may assume now $\{\tau(g): x(g)=1\} \tm \mathrm{int}(\widetilde{W}^{(k)}\xi^{-1})\cup S_k^\xi$ and there exists some $l\in G$ such that $\tau(l)\in S_k^\xi$ and $x(l) =1$.
Suppose that there exist distinct elements $\tau(l_1),\tau(l_2)\in S_k^\xi$ such that $x(l_1) = x(l_2)=0$.
In this case, we obtain for the set $N = \{l\}$ and $M = \{l_1,l_2\}$ by Lemma~\ref{lem: order structure tilde(W)} (4) that
$$ B_{\varepsilon}(\xi) \cap \kl \bigcap_{l\in N} \tau(l)^{-1}  \widetilde{W}^{(k)}\setminus \bigcup_{j\in M}\tau(j)^{-1} \widetilde{W}^{(k)}\kr$$
is empty for sufficiently small $\varepsilon>0$.
Therefore, Proposition~\ref{prop: criterion being in fiber} yields $x\notin \widetilde{\beta}^{-1}(\{\xi\})$, a contradiction.
This shows that $\#\kl S_k^\xi\setminus \{\tau(g):x(g)=1\}\kr \in\{0,1\}$.
One case implies $x = \widetilde{x}_k$ whereas the other $x= \widetilde{x}_{k,j}$, where $\tau(j)$ is the unique element of $S_k^\xi\setminus \{\tau(g):x(g)=1\}$.
This shows that $\widetilde{W}^{(k)}$ satisfies \eqref{eq: new fibers}.

Now let $$\widetilde{H}_0=\{\xi\in \overleftarrow{G}: (\partial \widetilde{W}^{(k)} \cap H_j)\xi^{-1}\cap \tau(G) \text{ is infinite for all } 1\leq j\leq k\}.$$
We have seen (after Corollary~\ref{cor: fiber structure for W^(k)}) that the pointwise ergodic theorem yields $\nu(\widetilde{H}_0) =1$. 
By definition, if $\xi\in \widetilde{H}_0$, then all the sets $S_j^\xi$, $1\leq j\leq k$ and $S_{k,l}^\xi$, $\tau(l)\in S_k^\xi$ defined above are non-empty.
It readily follows that all the elements $\widetilde{x}_j$ and $\widetilde{x}_{k,l}$ are distinct and the set $\{\widetilde{x}_{k,l}: \tau(l)\in S_k^\xi\}$ is (countably) infinite.
We claim that for $\xi\in \widetilde{H}_0$, the inclusion in \eqref{eq: new fibers} is an equality, which would conclude the proof of this Proposition and hence of Theorem~\ref{thm: existence of k-1 Toeplitz arrays with infinite max rank}.

By the above, we obtain that the set $\mathcal{S}$ of similarity classes w.r.t.\ $\xi$ is given by $\mathcal{S} =\{S_1^\xi,\ldots, S_{k-1}^\xi\} \cup \{S_{k,l}^\xi: \tau(l)\in S_k^\xi\}$.
Note that the similarity classes $S_{k,l}^\xi$ are all distinct by Lemma~\ref{lem: order structure tilde(W)} (3). 
We still have $S_{i_1}^\xi\prec_\xi S_{i_2}^\xi$ if $i_1<i_2<k$ and also $S_j^\xi\prec_\xi S_{k,l}^\xi$ whenever $j<k$ and $\tau(l)\in S_k^\xi$ by  Lemma~\ref{lem: order structure tilde(W)} (2). 

The fact that the elements $\widetilde{x}_j$, defined in \eqref{eq:tilde-x} with $j\neq k$, are in $\widetilde{\beta}^{-1}(\{\xi\})$ follows from Lemma~\ref{lem: partition yields fiber element} as follows:
If $j<k$, choose $\mathcal{S}^{-} = \{S_1^\xi,\ldots,S_{j-1}^\xi\}$ and $\mathcal{S}^{+} = \{S_j^\xi,\ldots,S_{k-1}^\xi\}\cup\{S_{k,l}^\xi:\tau(l)\in S_k^\xi\}$.
If $j=k+1$, then let $\mathcal{S}^{-}=\mathcal{S}$ and $\mathcal{S}^{+}=\emptyset$.

Now let $l_1\in G$ be such that $\tau(l_1)\in S_k^\xi$.
To show that $\widetilde{x}_{k,l_1}$ lies in the fiber, let $N,M\tm G$ be non-empty finite sets such that $\tau(N)\cup \tau(M)\tm \partial(\widetilde{W}^{(k)}\xi^{-1})$, $\widetilde{x}_{k,l_1}(g)=1$ for all $g\in N$ and $\widetilde{x}_{k,l_1}(g) = 0$ for all $g\in M$.
Then, clearly $M=\{l_1\}$.
However, the construction (in particular again  the fact that at most one cylinder set of every level is removed) yields that 
\begin{equation}\label{eq: tilde(x)_(k,l) in fiber}
 B_\varepsilon(\xi) \cap \bigcap_{l\in N} \tau(l)^{-1}\widetilde{W}^{(k)}\setminus  \tau(l_1)^{-1}\widetilde{W}^{(k)}
\end{equation}
has non-empty interior for every $\varepsilon>0$. 
Thus, $\widetilde{x}_{k,l_1}\in \widetilde{\beta}^{-1}(\{\xi\})$ follows from Lemma~\ref{lem: suff crit being in fiber}.
Note that since $l_1$ with $\tau(l_1)\in S_k^\xi$ was arbitrary, we also obtain from \eqref{eq: tilde(x)_(k,l) in fiber} that $B_{\varepsilon}(\xi) \cap \bigcap_{l \in N} \tau(l)^{-1} \widetilde{W}^{(k)}$ has non-empty interior for every $\varepsilon>0$ and finite subset $N$ such that $\tau(N)\tm \partial(\widetilde{W}^{(k)}\xi^{-1})$.
Hence, Lemma~\ref{lem: suff crit being in fiber} also implies $\widetilde{x}_k\in \widetilde{\beta}^{-1}(\{\xi\})$, finishing the proof.
\end{proof}


\begin{thebibliography}{MMM}

\bibitem{Au88}
Auslander, Joseph.
\textit{Minimal flows and their extensions,} volume 153 of North-Holland Mathematics
Studies. North-Holland Publishing Co., Amsterdam, 1988. Notas de Matemática, 122. [Mathematical Notes]. 



\bibitem{BaLeMo07} Baake, Michael; Lenz, Daniel; Moody, Robert V.  {\it Characterization of model sets by dynamical systems.} Ergodic Theory Dynam. Systems 27 (2007), no. 2, 341--382.

\bibitem{BaLe04} Baake, Michael; Lenz, Daniel. {\it Dynamical systems on translation bounded measures: pure point dynamical and diffraction spectra.}  Ergodic Theory Dynam. Systems 24 (2004), no. 6, 1867--1893.

\bibitem{BaJäLe16} Baake, Michael; J\"ager, Tobias;  Lenz, Daniel. {\it Toeplitz flows and model sets.} 
Bull. Lond. Math. Soc. 48 (2016), no. 4, 691--698.

\bibitem{BjHaPo18} Bj\"orklund, Michael; Hartnick, Tobias;  Pogorzelski, Felix. \textit{Aperiodic order and spherical diffraction, I: auto-correlation of regular model sets}, Proc. Lond. Math. Soc. (3) {\bf 116} (2018), no.~4, 957--996.

\bibitem{BeOh07}
Benoist, Yves; Oh, Hee.
\textit{Equidistribution of rational matrices in their conjugacy classes.}
Geom. Funct. Anal.17(2007), no.1, 1–32.

\bibitem{BreHaJa25} Breitenb{\"u}cher, J.;  Haupt, L.;   J{\"a}ger, T. \textit{Multivariate mean equicontinuity for finite-to-one topomorphic extensions}, {\tt arXiv:2409.08707}.

\bibitem{CeCo18} Ceccherini-Silberstein, Tullio; Coornaert, Michel. \textit{Cellular automata and groups}. Cellular automata, 221--238, Encycl. Complex. Syst. Sci., Springer, New York, 2018.

\bibitem{Co06} Cortez, Mar\'{\i}a  Isabel. \textit{$\Z^d$-{T}oeplitz arrays.} Discrete Contin. Dyn. Syst. 15 (2006), no. 3, 859--881.

\bibitem{CeCoGo23} Cecchi Bernales, Paulina; Cortez, María Isabel; Gómez, Jaime.  \textit{Invariant measures of Toeplitz subshifts on non-amenable groups}. Ergodic Theory Dynam. Systems 44 (2024), no. 11, 3186--3215.

\bibitem{CoGo24}   Cortez, Mar\'{\i}a  Isabel;  G\'omez, Jaime.  \textit{Almost 1-1 extensions of Furstenberg-Weiss type and  test for amenability,}  {\tt arXiv:2403.06982}.

\bibitem{CoPe08}Cortez, María Isabel; Petite, Samuel. \textit{$G$-odometers and their almost one-to-one extensions}, J. Lond. Math. Soc. (2) {\bf 78} (2008), no.~1, 1--20.

\bibitem{CoPe14} Cortez, María Isabel; Petite, Samuel. \textit{Invariant measures and orbit equivalence for generalized Toeplitz subshifts}. Groups Geom. Dyn. 8 (2014), no. \textbf{4}, 1007--1045.

\bibitem{Dow88}
Downarowicz, Tomasz.
\textit{A minimal 0-1 subshift with noncompact set of ergodic measures.}
Probab. Theory Related Fields, 79(1):29--35, 1988.

\bibitem{Dow91}
Downarowicz, Tomasz.
\textit{The Choquet simplex of invariant measures for minimal flows.}
Israel J. Math., 74(2-3):241--256, 1991.

\bibitem{DL96}
Downarowicz, T.; Lacroix, Y.
\textit{A non-regular Toeplitz flow with preset pure point spectrum.}
Studia Math., 120:235--246, 1996.

\bibitem{Do05} Downarowicz, Tomasz. \textit{Survey of odometers and {T}oeplitz flows.} Algebraic and topological dynamics, 7--37. Contemp. Math., 385 American Mathematical Society, Providence, RI, 2005.

\bibitem{DJL} Drewlo, Jamal; J{\"a}ger, Tobias;  Lenz, Daniel. \textit{A survey on model sets}, (2025) in preparation.

\bibitem{FuGlJäOe21} Fuhrmann, G., Glasner, E., J\"ager, T. and Oertel, C. {\it Irregular model sets and tame dynamics.} Trans. Amer. Math. Soc. 374 (2021), no. 5, 3703--3734.

\bibitem{FK18} Fuhrmann, Gabriel; Kwietniak, Dominik. 
\textit{On tameness of almost automorphic dynamical systems for general groups.} 
Bull. Lond. Math. Soc. 52 (2020), no. 1, 24--42.

\bibitem{GoLeMu25} Gómez, Jaime; León-Torres, Irma; Muñoz-López, Victor. \textit{Sequence entropy and independence in free and minimal actions}, {\tt arXiv:2504.00960}.

\bibitem{HofMo20} Hofmann, Karl H.; Morris, Sidney A. \textit{The structure of compact groups—a primer for the student— a handbook for the expert}, volume 25 of De Gruyter Studies in Mathematics. De Gruyter,
Berlin, fourth edition, [2020] ©2020.


\bibitem{HLSY21} Huang, Wen; Lian, Zhengxing; Shao, Song; Ye, Xiangdong. \textit{Minimal systems with finitely many ergodic measures.} J. Funct. Anal. 280 (2021), no. 12, Paper No. 109000, 42 pp.

\bibitem{HuYe09} Huang, Wen; Ye, Xiangdong  \textit{Combinatorial lemmas and applications to dynamics}. Adv.
Math. 220 (2009), no. 6, 1689–1716.

\bibitem{IL94}
Iwanik, A.; Lacroix, Y.
{\it Some constructions of strictly ergodic non-regular Toeplitz flows}.
Studia Math. 110 (1994), no. 2, 191--203.

\bibitem{JK69} Jacobs, Konrad; Keane, Michael.
{\it $0$-$1$-sequences of Toeplitz type}.
Z. Wahrscheinlichkeitstheorie und Verw. Gebiete 13 (1969), 123--131.

\bibitem{JLO19} J{\"a}ger, Tobias; Lenz, Daniel; Oertel, Christian. {\it Model sets with positive entropy in Euclidean cut and project schemes.} Ann. Sci. Éc. Norm. Supér. (4) 52 (2019), no. 5, 1073--1106.

\bibitem{KL16} Kerr, David; Li, Hanfeng.
{\it Ergodic theory. Independence and dichotomies}.
Springer Monogr. Math.
Springer, Cham, 2016.

\bibitem{Kr10} Krieger, F.
{\it Sous-d\'ecalages de Toeplitz sur les groupes moyennables r\'esiduellement finis}.
J. Lond. Math. Soc. (2) 75 (2007), no. 2, 447--462. 

\bibitem{LP03}
Lagarias, Jeffrey C.; Pleasants, Peter A. B. {\it Repetitive Delone sets and quasicrystals.}
Ergodic Theory Dynam. Systems 23 (2003), no. 3, 831--867.

\bibitem{LaStra18} Łącka, Martha; Straszak, Marta. \textit{Quasi-uniform convergence in dynamical systems generated by an amenable group action.} J. Lond. Math. Soc. (2) 98 (2018), no. 3, 687--707.

\bibitem{LS03}
Lenz, Daniel; Stollmann, Peter.
{\it Delone dynamical systems and associated random operators.}  Operator algebras and mathematical physics (Constanţa, 2001), 267--285. Theta, Bucharest, 2003.

\bibitem{Lin01}
Lindenstrauss, Elon. 
{\it Pointwise theorems for amenable groups.} 
Invent. Math. 146 (2001), no. 2, 259--295.

\bibitem{MP79}
Markley, Nelson G.; Paul, Michael E.
{\it Almost automorphic symbolic minimal sets without unique ergodicity.}
Israel J. Math. 34 (1979), no. 3, 259--272.

\bibitem{Mey72} Meyer, Yves. {\it Algebraic numbers and harmonic analysis.} North-Holland Math. Library, Vol. 2 North-Holland Publishing Co., Amsterdam-London; American Elsevier Publishing Co., Inc., New York, 1972.  

\bibitem{Mo97} Moody, Robert V. 
{\it Meyer sets and their duals.} The mathematics of long-range aperiodic order (Waterloo, ON, 1995), 403--441. NATO Adv. Sci. Inst. Ser. C: Math. Phys. Sci., 489 Kluwer Academic Publishers Group, Dordrecht, 1997.

\bibitem{Ox52} Oxtoby, John C. {\it Ergodic sets.} Bull. Amer. Math. Soc. 58 (1952), 116--136. 

\bibitem{Sol98}
Solomyak, Boris.
{\it Spectrum of dynamical systems arising from Delone sets}. Quasicrystals and discrete geometry (Toronto, ON, 1995), 265--275. Fields Inst. Monogr., 10 American Mathematical Society, Providence, RI, 1998.

\bibitem{Wi84} Williams, Susan.
{\it Toeplitz minimal flows which are not uniquely ergodic}.
Z. Wahrsch. Verw. Gebiete 67 (1984), no. 1, 95--107.







\end{thebibliography}
\end{document}